\documentclass{imanum}
	\usepackage{graphics}
	\usepackage{tikz}
\newcommand{\vertiii}[1]{{\left\vert\kern-0.25ex\left\vert\kern-0.25ex\left\vert #1 
    \right\vert\kern-0.25ex\right\vert\kern-0.25ex\right\vert}}

\received{13 October 2014}
\revised{06 July 2015}
\begin{document}

\title{A penalty free Nitsche method for the weak imposition of boundary conditions in compressible and incompressible elasticity}
% Short title for running heads:
\shorttitle{Penalty free Nitsche method, elasticity}

\author{%
{\sc
Thomas Boiveau\thanks{Corresponding author. Email: thomas.boiveau.12@ucl.ac.uk}
and
Erik Burman\thanks{Email: e.burman@ucl.ac.uk}
} \\[2pt]
Department of Mathematics, University College London,\\Gower Street, London, WC1E  6BT, UK
}
% Short list of authors for running heads:
\shortauthorlist{T. Boiveau and E. Burman}

\maketitle

\begin{abstract}
% Body of abstract:
{In this paper, we study the stability of the nonsymmetric version of Nitsche's method without penalty for compressible and incompressible elasticity. For the compressible case we prove the convergence of the error in the $H^1$- and $L^2$-norms. In the incompressible case we use a Galerkin least squares pressure stabilization and we prove the convergence in the $H^1$-norm for the velocity and convergence of the pressure in the $L^2$-norm.}
% Keywords:
{Nitsche's method; compressible elasticity; inompressible elasticity; stabilized finite element methods; Korn inequality.}
\end{abstract}

\section{Introduction}
In the seminal paper of \cite{Nitsche_1971_a}, a consistent penalty method for the weak imposition of boundary conditions was introduced.
The method relied on a penalty term, the parameter of which had to be sufficiently large in order for stability to be ensured.
\cite{Freund_1995_a} then suggested a nonsymmetric version of Nitsche's method. The advantage of the nonsymmetric version was that no lower bound had to be respected for the penalty parameter, it only needed to be strictly larger than zero. The symmetric and nonsymmetric versions of Nitsche's method were further discussed by \cite{Hughes_2000_a}, where the possibility of using the nonsymmetric version with zero penalty parameter was mentioned. Penalty free nonsymmetric methods have indeed been advocated for the discontinuous Galerkin method \citep[see,][]{Oden_1998_a, Larson_2004_a, Girault_2009_a, Burman_2010_c}. \cite{Burman_2012_b} proved that the nonsymmetric Nitsche method was stable without penalty for scalar elliptic problems. The main observation in that paper was that although coercivity fails for the bilinear form when the penalty parameter was set to zero, the formulation could be proven to be inf-sup stable. Using the discrete stability optimal error estimates were obtained in the energy norm.

The nonsymmetric version of Nitsche's method without penalty can be seen as a Lagrange multiplier method, where the Lagrange multiplier has been replaced by the boundary fluxes of the discrete elliptic operator. This leads to a method that is stable without any unknown parameter and without introducing additional degrees of freedom. Eliminating the penalty term appears to have some advantages in multi-physics coupling problems in elasticity, \citep[see for instance,][]{Burman_2014_c} and it is therefore interesting to understand the structure and stability mechanisms of the method in such a context.

In this paper we extend the results of \cite{Burman_2012_b} to the case of the equations of linear elasticity. Both the cases of compressible and incompressible elasticity are considered. The main difficulties compared to the scalar case are:
\begin{itemize}
\item the Nitsche boundary term is no longer based on the gradient but now contains the deformation tensor and the divergence;
\item it is no longer clear that Korn's inequality holds;
\item for incompressible elasticity the inf-sup condition must be shown to hold simultaneously for the boundary conditions and the pressure.
\end{itemize}
We end this section by introducing the models of compressible and incompressible elasticity.
Let $\Omega$ be a convex bounded domain in $\mathbb{R}^2$, with polygonal boundary $\partial\Omega$. This boundary is decomposable such that $\partial \Omega = \cup_i\Gamma_i$ with $\{\Gamma_i\}_i$ the sides of the polygonal. $\boldsymbol{f} \in \left[L^2(\Omega)\right]^2$ is a given body force and $\boldsymbol{g}\in \left[H^{1/2}\left(\Omega\right)\right]^2$ the value of $\boldsymbol{u}$ at the boundary.

Compressible elasticity: find the displacement $\boldsymbol{u} : \Omega \subset \mathbb{R}^2 \rightarrow \mathbb{R}^2$ such that
\begin{eqnarray}
\label{elasticity}
-\nabla \cdot \boldsymbol{\sigma}(\boldsymbol{u})&=& \boldsymbol{f}  ~~~~\mbox{ in } \Omega \nonumber, \\
\boldsymbol{u} &=& \boldsymbol{g}  ~~~~\mbox{ on } \partial\Omega, 
\end{eqnarray}

with
$$\boldsymbol{\sigma}(\boldsymbol{u}):=2\mu\boldsymbol{\varepsilon}(\boldsymbol{u})+\lambda\left( \nabla \cdot \boldsymbol{u}\right)\mathbb{I}_{2\times 2}.$$

Incompressible elasticity: find the velocity $\boldsymbol{u} : \Omega \subset \mathbb{R}^2 \rightarrow \mathbb{R}^2$ and the pressure $p : \Omega \rightarrow \mathbb{R}$ such that
\begin{eqnarray}
\label{stokes}
-\nabla \cdot \boldsymbol{\sigma}\left(\boldsymbol{u},p\right)&=& \boldsymbol{f}  ~~~~\mbox{ in } \Omega \nonumber, \\
\nabla \cdot \boldsymbol{u}&=& 0 ~~~~\mbox{ in } \Omega,\\
\boldsymbol{u} &=& \boldsymbol{g}  ~~~~\mbox{ on } \partial\Omega \nonumber,
\end{eqnarray}

with
$$\boldsymbol{\sigma}(\boldsymbol{u},p):=2\mu\boldsymbol{\varepsilon}(\boldsymbol{u})+p\mathbb{I}_{2\times 2}.$$

To ensure the divergence free property of the incompressible case we assume $\int_{\partial\Omega}\boldsymbol{g}\cdot\boldsymbol{n}~\text{d}x=0$ where $\boldsymbol{n}$ denotes the outward normal vector of the boundary. For future reference we introduce the function spaces $V:=\left[H^1(\Omega)\right]^2$, $V_0:=\left[H^1_0(\Omega)\right]^2$ and $Q:=\{p\in L^2\left(\Omega\right),~\int_\Omega p~\text{d}x=0\}$.

\section{Preliminaries}
\label{preliminaries}
The set $\{\mathcal{T}_h\}_h$ defines a family of quasi-uniform and shape regular triangulations fitted to $\Omega$. We define the shape regularity as the existence of a constant $c_\rho\in\mathbb{R}_+$ for the family of triangulations such that, with $\rho_K$ the radius of the largest circle inscribed in an element $K$, there holds
$$\frac{h_K}{\rho_K}\leq c_\rho~~~~\forall K \in \mathcal{T}_h.$$ 

In a generic sense we define $K$ as the triangles in a triangulation  $\mathcal{T}_h$ and $h_K:=\mbox{diam}(K)$ is the diameter of $K$. Then we define $h:=\mbox{max}_{K\in\mathcal{T}_h}h_K$ as the mesh parameter for a given triangulation $\mathcal{T}_h$. $\mathbb{P}_k(K)$ defines the space of polynomials of degree less than or equal to $k$ on the element $K$. We define $V_h^k$ and $Q_h^k$ the finite element spaces of continuous piecewise polynomial functions
\begin{eqnarray*}
V_h^k&:=&\left\{\boldsymbol{u}_h\in V:\boldsymbol{u}_h|_K\in\left[\mathbb{P}_k\left(K\right)\right]^2~~\forall K\in\mathcal{T}_h\right\},~~k\geq 1,\\
Q_h^k&:=&\left\{p_h\in Q:p_h|_K\in\mathbb{P}_k\left(K\right)~~\forall K\in\mathcal{T}_h\right\},~~k\geq 1.
\end{eqnarray*}
For simplicity we will write the $L^2$-norm on a domain $\Theta$, $\left\|\cdot\right\|_{L^2\left(\Theta\right)}$ as $\left\|\cdot\right\|_\Theta$. In this paper $C$ will be used as a generic positive constant that may change at each occurrence, we will use the notation $a\lesssim b$ for $a\leq C b$. We now recall several classical inequalities and various mathematical concepts.
\begin{lemma}
\label{trace}
There exists $C_T\in \mathbb{R}_+$ such that for all $u\in H^1\left(K\right)$ and for all $K\in\mathcal{T}_h$, the trace inequality holds
$$\left\|u\right\|_{\partial K}\leq C_T \left(h_K^{-\frac12}\left\|u\right\|_{K}+h_K^{\frac12}\left\|\nabla u\right\|_{K}\right).$$
\end{lemma}
\begin{lemma}
\label{inverse}
There exists $C_I\in \mathbb{R}_+$ such that for all $u_h\in\mathbb{P}_k(K)$ and for all $K\in\mathcal{T}_h$, the inverse inequality holds
$$\left\|\nabla u_h\right\|_{K}\leq C_I h_K^{-1}\left\|u_h\right\|_{K}.$$
\end{lemma}
%\begin{definition}
%Let $\hat{K}$ be the reference $d$-simplex and $K=G(\hat{K})$ under the affine mapping $G:\mathbb{R}^d\rightarrow\mathbb{R}^d$ considering $G(\hat{\boldsymbol{z}})=A\hat{\boldsymbol{z}}+b$ with $A\in\mathbb{R}^{d\times d}$ and $b\in\mathbb{R}^d$.
%Then we can define the \textbf{Piola transform} $\mathcal{P}$ such that
%\begin{eqnarray*}
%\mathcal{P}:L^2( \hat{K})^d & \longrightarrow & L^2(K)^d\\
%\hat{\boldsymbol{z}} &\longmapsto & \boldsymbol{z}=\mathcal{P}(\hat{\boldsymbol{z}}):=\frac{1}{det(A)}A\hat{\boldsymbol{z}}.
%\end{eqnarray*}
%\end{definition}

Anticipating the inf-sup analysis of the coming section we introduce patches of boundary elements for the construction of special functions in the finite element space $V_h^k$ that will serve for the proof of stability. We will first detail the geometric construction and then give a technical Lemma that is needed in the coming analysis. We regroup the boundary elements in closed, disjoint patches $P_j$ with boundary $\partial P_j$, $j=1,...,N_{P}$. $N_P$ defines the total number of patches. The boundary elements are the elements with either a face or a vertex on the boundary. Every boundary element is a member of exactly one patch $P_j$. The number of elements necessary in each patch
is always at least two and upper bounded by a constant depending only on the shape regularity parameter $c_\rho$. Let $F_j:=\partial P_j \cap \partial \Omega$, we assume that every $\Gamma_i$ is partitioned by at least one $F_j$. Define the boundary elements by $P:=\cup_j P_j$. For each $F_j$ there exists two positive constants $c_1$, $c_2$ such that for all $j$
$$c_1h\leq\mbox{meas}(F_j)\leq c_2 h.$$
Figure \ref{patch_draw} gives a representation of a patch as defined above with four inner nodes.
Let $\phi_j \in V_h^1$ be defined for each node $r_i\in\mathcal{T}_h$ such that for each patch $P_j$
\begin{displaymath}
\phi_j\left(r_i\right)=
\left\{
\begin{array}{rllr}
0&\text{for}& r_i \in \Omega \backslash \mathring{F}_j\\
0&\text{for}& x_i \in K \text{ such that } K \text{ has all its vertices on } \partial\Omega\\
1&\text{for}& r_i \in \mathring{F}_j,
\end{array}\right.
\end{displaymath}
with $i=1,\dots,N_n$. Here $N_n$ is the number of nodes in the triangulation $\mathcal{T}_h$ and $\mathring{F}_j$ defines the interior of the face $F_j$.
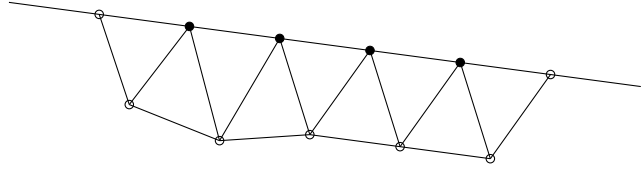
\begin{figure}[h!]
    \begin{center}
\begin{tikzpicture}[scale=0.8]

    \draw (-5,4.2)  -- (5.5,2.8);
    \draw (-3,2.5)  -- (-1.5,1.9);
    \draw (-1.5,1.9)  -- (0,2);
    \draw (0,2)  -- (3,1.6);

\draw (-3.5,4)  -- (-3,2.5);
\draw (3,1.6)  -- (4,3);
\draw (3,1.6)  -- (2.5,3.2);
\draw (1.5,1.8)  -- (2.5,3.2);
\draw (1.5,1.8)  -- (1,3.4);
\draw(0,2)  -- (1,3.4);
\draw(0,2)  -- (-0.5,3.6);
\draw(-1.5,1.9)  -- (-0.5,3.6);
\draw(-1.5,1.9)  -- (-2,3.8);
\draw(-3,2.5)  -- (-2,3.8);

%    \draw (-3.5,4)  -- (-5.5,3);
%    \draw (5.5,2)  -- (4,3);

\draw[] (-3.5,4) circle(2pt);
\draw[] (4,3) circle(2pt);
\draw[] (-3,2.5) circle(2pt);
\draw[fill=black] (-2,3.8) circle(2pt);
\draw[fill=black] (-0.5,3.6) circle(2pt);
\draw[fill=black] (1,3.4) circle(2pt);
\draw[fill=black] (2.5,3.2) circle(2pt);
\draw[] (-1.5,1.9) circle(2pt);
\draw[] (0,2) circle(2pt);
\draw[] (1.5,1.8) circle(2pt);
\draw[] (3,1.6) circle(2pt);

%	\draw [->] (0,3.7)--(0.9912,3.5677) node[above left]  {$\xi$};
%	\draw[->] (0,3.7)--(0.1323,4.6912) node[below left] {$\eta$};

%	\draw [->] (-4,0)--(-3,0) node[below left]  {$x$};
%	\draw[->] (-4,0)--(-4,1) node[below left] {$y$};

\end{tikzpicture}
\end{center}
\label{patch_draw}
\caption{Example of a patch $P_j$, the function $\boldsymbol{\phi}_j$ is equal to $0$ in the nonfilled nodes, $1$ in the filled nodes.}
\end{figure}

We define the function $\boldsymbol{v}_h\in V_h^k$ such that $\boldsymbol{v}_h:=\boldsymbol{u}_h+\boldsymbol{v}_\Gamma$, with $\boldsymbol{u}_h, \boldsymbol{v}_\Gamma\in V_h^k$. The function $\boldsymbol{v}_\Gamma$ is defined such that
\begin{equation}
\label{defvgamma}
\boldsymbol{v}_\Gamma=\sum_{j=1}^{N_P} \boldsymbol{v}_j=\sum_{j=1}^{N_P}\left(\alpha_1v_{j1},\alpha_2v_{j2}\right)^{\rm T},
\end{equation}
with
\begin{equation}
\label{defvj}
v_{j1}=\zeta_{j1}\phi_j~~~,~~~~~v_{j2}=\zeta_{j2}\phi_j~~~,~~~~~\zeta_{j1}, \zeta_{j2}\in \mathbb{R},
\end{equation}
for simplicity of notation we will use $v_1$, $v_2$ respectively instead of $v_{j1}$, $v_{j2}$. To define the properties of $v_1$ and $v_2$ we need to introduce the projection of $u$ on constant functions on the interval $I$
$$P_0 u\vert_{I}:=\text{meas}\left(I\right)^{-1}\int_{I} u~\text{d}s.$$
For simplicity of notation we will also use the notation $\overline{u}^j:=P_0 u\vert_{F_j}$. We introduce the following two dimensional rotation transformation.
\begin{definition}
The rotation transformation in two dimensions can be written as
\begin{eqnarray*}
\mathcal{R}:\left[L^2\left(\hat{\Omega}\right)\right]^2 & \longrightarrow & \left[L^2\left(\Omega\right)\right]^2\\
\hat{\boldsymbol{z}} &\longmapsto & \boldsymbol{z}=\mathcal{R}(\hat{\boldsymbol{z}}):=A\hat{\boldsymbol{z}},
\end{eqnarray*}
with $A$ a rotation matrix and $\hat{\boldsymbol{z}}$ the rotated quantity of $\boldsymbol{z}$.
\end{definition}
This two-dimensional rotation is used to transform the generic fixed frame $(x,y)$ into a rotated frame $(\xi,\eta)$ associated to each side $\Gamma_i$ of $\partial \Omega$. This rotated frame has its first component tangent to the side $\Gamma_i$ of the polygonal boundary and its second component normal to this same side $\Gamma_i$. Defining $\boldsymbol{\tau}$ as the unit tangent vector to the boundary, a function $\boldsymbol{z}=(z_1,z_2)$ expressed in the two-dimentional rotated frame has the following properties
$$\hat{z}_1=\boldsymbol{z}\cdot \boldsymbol{\tau}~~,~~~~~~~\hat{z}_2=\boldsymbol{z}\cdot \boldsymbol{n}.$$

The hat denotes a value expressed in the rotated frame $(\xi,\eta)$. Figure \ref{face_rotation_frame} represents schematically how is defined this frame for a side $\Gamma_i$.
\begin{figure}[h!]
    \begin{center}
\begin{tikzpicture}[scale=0.8]

    \draw (-3.5,4)  -- (4,3);

    \draw (-3.5,4)  -- (-5.5,3);
    \draw (5.5,2)  -- (4,3);

	\draw [->] (0,3.7)--(0.9912,3.5677) node[above left]  {$\xi$};
	\draw[->] (0,3.7)--(0.1323,4.6912) node[below left] {$\eta$};

	\draw [->] (-2,2)--(-1,2) node[below left]  {$x$};
	\draw[->] (-2,2)--(-2,3) node[below left] {$y$};
	
\draw (-1.5,3.8)node[above left]{$\Gamma_i$};
\draw (-4.3,3.5)node[above left]{$\Gamma_{i-1}$};
\draw (6,2.3)node[above left]{$\Gamma_{i+1}$};

\end{tikzpicture}
\end{center}
\label{face_rotation_frame}
\caption{Representation of the rotated frame $(\xi,\eta)$, the first component of the frame is tangent to the side $\Gamma_i$ and the second component is normal to the side $\Gamma_i$.}
\end{figure}
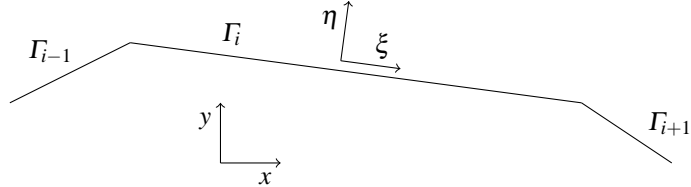

Using the rotation transformation $\hat{\boldsymbol{u}}_h=\left(\hat{u}_1,\hat{u}_2\right)^{\rm T}$, we may now define $v_1$ and $v_2$ by the relations
\begin{equation}
\label{prop_vj}
\text{meas}\left(\hat{F}_j\right)^{-1}\int_{\hat{F}_j}\frac{\partial \hat{v}_1}{\partial \eta}~\text{d}\hat{s}:=P_0\hat{u}_1\vert_{\hat{F}_j}~~,~~~~~~\text{meas}\left(\hat{F}_j\right)^{-1}\int_{\hat{F}_j}\frac{\partial \hat{v}_2}{\partial \eta}~\text{d}\hat{s}:=P_0\hat{u}_2\vert_{\hat{F}_j}.
\end{equation}
\begin{lemma}
\label{prop_patch}
Let $P_j$ be a patch and $\boldsymbol{v}_j$ a function as defined above, $\forall \boldsymbol{u}_h\in V_h^k$ the following inequalities are true
\begin{eqnarray}
\left\|\boldsymbol{u}_h-\overline{\boldsymbol{u}}_h^j \right\|_{F_j}&\lesssim& h\left\|\nabla \boldsymbol{u}_h\cdot \boldsymbol{\tau}\right\|_{F_j},\label{stdapprox}\\
\left\|h^{-\frac12}\boldsymbol{u}_h\right\|_{F_j}^2-C\left\|\nabla\boldsymbol{u}_h\right\|_{P_j}^2&\leq&\left\|h^{-\frac12}\overline{\boldsymbol{u}}_h^j\right\|_{F_j}^2,\label{inequality4}\\
\left\|\boldsymbol{v}_j\right\|_{P_j}&\lesssim& h\left\|\nabla \boldsymbol{v}_j\right\|_{P_j},\label{inequality7}\\
\left\|\nabla \hat{v}_1\right\|_{\hat{P}_j}&\leq& C\left\| h^{-\frac12}\overline{\boldsymbol{u}}_h^j \cdot \boldsymbol{\tau}\right\|_{F_j},\label{keyineq1}\\
\left\|\nabla \hat{v}_2\right\|_{\hat{P}_j}&\leq& C\left\| h^{-\frac12}\overline{\boldsymbol{u}}_h^j \cdot \boldsymbol{n}\right\|_{F_j}.\label{keyineq2}
\end{eqnarray}
The constant in (\ref{keyineq1}), (\ref{keyineq2}) is bounded uniformly provided each patch contains a sufficient number of elements
compared to $c_\rho$.
\end{lemma}
\begin{proof}
 See Appendix.
\end{proof}

In the analysis, we will need a particular form of Korn's inequality. To prove this alternative form of the Korn's inequality we need to define first the following seminorm
\begin{equation}
\label{seminorm}
\left|\boldsymbol{u}\right|_{\Gamma}^2:=\sum_{i=1}^{N_b}\int_{\Gamma_i}\left(P_{0}\boldsymbol{u}\right)^2~\textup{d}s~~~~~~~\forall\boldsymbol{u}\in V,
\end{equation}
with $\Gamma_i$ the $i^{th}$ side of the polygonal boundary $\partial \Omega$, $i=1,...,N_b$, $N_b$ is the number of sides on the boundary. $P_{0}\boldsymbol{u}\vert_{\Gamma_i}$ is the $P_0$-projection of $\boldsymbol{u}$ on the side $\Gamma_i$.
\begin{proposition}
 \label{norm_korn}
For all $\boldsymbol{u}\in V$ the seminorm (\ref{seminorm}) is a norm on $\textup{RM}$ with
$$\textup{RM}:=\left\{ \boldsymbol{u}:\boldsymbol{u}=\boldsymbol{c}+b\left( x_2,-x_1\right)^{\rm T},\boldsymbol{c}\in \mathbb{R}^2, b\in \mathbb{R}\right\}.$$
\end{proposition}
\begin{proof}
The claim follows from direct inspection of the linear system resulting from $P_0 \boldsymbol{u}\vert_{\Gamma_i} =0$.
\end{proof}
%Since $\left|\boldsymbol{u}\right|_\Gamma$ is a seminorm, we only need to show that $\left|\boldsymbol{u}\right|_{\Gamma}= 0\Rightarrow\boldsymbol{u}=0$ $\forall \boldsymbol{u}\in \text{RM}$. First note that
%$$\left|\boldsymbol{u}\right|_{\Gamma}= 0\Rightarrow\sum_{i=1}^{N_b}\int_{\Gamma_i}\left(P_{0}\boldsymbol{u}\right)^2~\text{d}s=0 \Rightarrow P_{0}\boldsymbol{u}\vert_{\Gamma_i}=0.$$
%We also know that $\boldsymbol{u}\in\text{RM}$ then
%$$\boldsymbol{u}=\begin{pmatrix}c_1+bx_2\\ c_2-bx_1\end{pmatrix},~~~~~~~~~~~P_{0}\boldsymbol{u}\vert_{\Gamma_i}=\begin{pmatrix}c_1+bP_{0}\left(x_2\right)\vert_{\Gamma_i}\\ c_2-bP_{0}\left(x_1\right)\vert_{\Gamma_i}\end{pmatrix}.$$
%Considering two sides $\Gamma_1$ and $\Gamma_2$ we obtain
%\begin{eqnarray*}
%c_1+bP_{0}\left(x_2\right)\vert_{\Gamma_1}&=&0,\\
%c_2-bP_{0}\left(x_1\right)\vert_{\Gamma_1}&=&0,\\
%c_1+bP_{0}\left(x_2\right)\vert_{\Gamma_2}&=&0,\\
%c_2-bP_{0}\left(x_1\right)\vert_{\Gamma_2}&=&0.
%\end{eqnarray*}
%Observe that we only need $P_{0}(x_1)\vert_{\Gamma_1}\ne P_{0}(x_1)\vert_{\Gamma_2}$ or $P_{0}(x_2)\vert_{\Gamma_1}\ne P_{0}(x_2)\vert_{\Gamma_2}$ to obtain $c_1=c_2=b=0$. One of these condition is always true since $\left(P_{0}\left(x_1\right)\vert_{\Gamma_1};P_{0}\left(x_2\right)\vert_{\Gamma_1}\right)$ and $\left(P_{0}\left(x_1\right)\vert_{\Gamma_2};P_{0}\left(x_2\right)\vert_{\Gamma_2}\right)$ are respectively the mid-points of the sides $\Gamma_1$ and $\Gamma_2$. This implies
%$$\left|\boldsymbol{u}\right|_{\Gamma}= 0
%\Rightarrow\boldsymbol{u}=0.$$

The alternative form of the Korn's inequality which will allow us to control the deformation tensor is expressed in the following theorem.
\begin{theorem}
\label{korn}
There exists a positive constant $C_K$ such that $\forall \boldsymbol{u}\in V$
$$C_K\left\|\boldsymbol{u}\right\|_{H^1(\Omega)}\leq\left\|\boldsymbol{\varepsilon}\left( \boldsymbol{u}\right)\right\|_{\Omega}+\left|\boldsymbol{u}\right|_{\Gamma}.$$
\end{theorem}
\begin{proof}
This proof is inspired by the proof of the Korn's inequality in \cite{Brenner_2008_a}. First we define $\tilde{V}$
$$\tilde{V}:=\left\{ \boldsymbol{u}\in V:\int_\Omega \boldsymbol{u}~\text{d}x=0, \int_\Omega \text{rot}~\boldsymbol{u}~\text{d}x=0\right\}.$$
We know that, $V=\tilde{V} \oplus \text{RM}$. Therefore, given any $\boldsymbol{u}\in V$, there exists a unique pair $\left( \boldsymbol{z}, \boldsymbol{w}\right)\in \tilde{V}\times \text{RM}$ such that
$$\boldsymbol{u}=\boldsymbol{z}+\boldsymbol{w}.$$
By the Open Mapping Theorem (Theorem 15, chapter 15 of \cite{Lax_2002_a}) there exists a positive constant $C_1$ such that
\begin{equation}
C_1\left(\left\|\boldsymbol{z}\right\|_{H^1(\Omega)}+\left\|\boldsymbol{w}\right\|_{H^1(\Omega)}\right)\leq\left\|\boldsymbol{u}\right\|_{H^1(\Omega)}.
\label{triangle}
\end{equation}
We establish the theorem by contradiction. If the inequality that we want to show does not hold for any positive constant $C_K$, then there exists a sequence $\{\boldsymbol{u}_n\}\subseteq V$ such that
\begin{equation}
\left\|\boldsymbol{u}_n\right\|_{H^1(\Omega)}=1,
\label{hyp2}
\end{equation}
and
\begin{equation}
\left\|\boldsymbol{\varepsilon}\left( \boldsymbol{u}_n\right)\right\|_{\Omega}+\left|\boldsymbol{u}_n\right|_{\Gamma}<\frac{1}{n}.
\label{hyp3}
\end{equation}
For each $n$, let $\boldsymbol{u}_n=\boldsymbol{z}_n+\boldsymbol{w}_n$, where $\boldsymbol{z}_n\in\tilde{V}$ and $\boldsymbol{w}_n\in\text{RM}$, then
$$\left\|\boldsymbol{\varepsilon}\left( \boldsymbol{z}_n\right)\right\|_{\Omega}=\left\|\boldsymbol{\varepsilon}\left( \boldsymbol{u}_n\right)\right\|_{\Omega}<\frac{1}{n}.$$
The second Korn's inequality then implies that $\boldsymbol{z}_n \longrightarrow 0$ in $V$. It follows from (\ref{triangle}) and (\ref{hyp2}) that $\{\boldsymbol{w}_n\}$ is a bounded sequence in $V$. But since $\text{RM}$ is finite dimensional, $\{\boldsymbol{w}_n\}$ has a convergent subsequence $\{\boldsymbol{w}_{n_j}\}$ in $V$. Then the subsequence $\{\boldsymbol{u}_{n_j}=\boldsymbol{z}_{n_j}+\boldsymbol{w}_{n_j}\}$ converges in 
$V$ to some $\boldsymbol{u}=\lim_{n_j \rightarrow \infty} \boldsymbol{w}_{n_j} \in \text{RM}$, we obtain
\begin{equation}
\left\|\boldsymbol{u}\right\|_{H^1(\Omega)}=1,
\label{contradiction3}
\end{equation}
and
\begin{equation*}
\left|\boldsymbol{u}\right|_{\Gamma}=0.
\end{equation*}
The Proposition \ref{norm_korn} tells us that $\left|\cdot\right|_\Gamma$ is a norm on $\text{RM}$ and therefore 
$$\left|\boldsymbol{u}\right|_{\Gamma}= 0
\Leftrightarrow\boldsymbol{u}=0,$$
which contradicts the equation (\ref{contradiction3}).
\end{proof}

\section{Compressible elasticity}
The first case that we consider is the compressible problem described by the system (\ref{elasticity}). We have the following weak formulation: find $\boldsymbol{u}\in V_g$ such that
$$a\left(\boldsymbol{u},\boldsymbol{v}\right)=\left(\boldsymbol{f},\boldsymbol{v}\right)_\Omega~~~~~~~~~\forall \boldsymbol{v}\in V_0,$$
where $\left(x,y\right)_\Omega$ is the $L^2$-scalar product over $\Omega$, $V_g:=\left\{\boldsymbol{v}\in \left[H^1\left(\Omega\right)\right]^2:\boldsymbol{v}|_{\partial \Omega}=\boldsymbol{g}\right\}$ and
$$a\left(\boldsymbol{u},\boldsymbol{v}\right)=\left(2\mu\boldsymbol{\varepsilon}(\boldsymbol{u}),\boldsymbol{\varepsilon}(\boldsymbol{v})\right)_\Omega+\left(\lambda \nabla \cdot \boldsymbol{u}, \nabla \cdot \boldsymbol{v} \right)_{\Omega}.$$
\subsection{Finite element formulation}
The nonsymmetric Nitsche's method applied to the compressible elasticity problem (\ref{elasticity}) leads to the following variational formulation, find $\boldsymbol{u}_h\in V_h^k$ such that
\begin{equation}
\label{formulationelasticity}
A_h\left(\boldsymbol{u}_h,\boldsymbol{v}_h\right)=L_h\left(\boldsymbol{v}_h\right)~~~~~~~~~\forall \boldsymbol{v}_h\in V_h^k,
\end{equation}
where the bilinear forms $A_h$ and $L_h$ are defined as
\begin{eqnarray*}
A_h(\boldsymbol{u}_h,\boldsymbol{v}_h)&=&a(\boldsymbol{u}_h,\boldsymbol{v}_h)-b(\boldsymbol{u}_h,\boldsymbol{v}_h)+b(\boldsymbol{v}_h,\boldsymbol{u}_h),\\
L_h(\boldsymbol{v}_h)&=&\left(\boldsymbol{f},\boldsymbol{v}_h\right)_\Omega+b(\boldsymbol{v}_h,\boldsymbol{g}).
\end{eqnarray*}
The bilinear form $b$ is defined as
\begin{eqnarray*}
b(\boldsymbol{u}_h,\boldsymbol{v}_h)&=&\left\langle 2\mu \boldsymbol{\varepsilon}\left(\boldsymbol{u}_h\right) \cdot \boldsymbol{n},\boldsymbol{v}_h \right\rangle_{\partial \Omega}+\left\langle\lambda \nabla \cdot \boldsymbol{u}_h,\boldsymbol{v}_h \cdot \boldsymbol{n} \right\rangle_{\partial \Omega}.
\end{eqnarray*}
% added review 1
In (\ref{formulationelasticity}), $a(\boldsymbol{u}_h,\boldsymbol{v}_h)$ represents the terms defined over the whole computational domain, $-b(\boldsymbol{u}_h,\boldsymbol{v}_h)$ is necessary for the consistency of the method, since $v_h \ne 0$, the antisymmetric contribution $b(\boldsymbol{v}_h,\boldsymbol{u}_h)$ and its corresponding term in $L_h$ together  impose the boundary condition.
\subsection{Stability}
The main goal of this section is to show the inf-sup condition. We first give two technical Lemmas, proofs are provided in Appendix.
\begin{lemma}
\label{inequality6}
There exists $C > 0$ independent of $h$, $\mu$ and $\lambda$, but not of the mesh geometry, $\forall\boldsymbol{u}_h\in V_h^k$, on each patch $P_j$ for $\boldsymbol{v}_j\in V_h^k$ as defined in equation (\ref{defvgamma}) and $\forall \epsilon,\alpha_1,\alpha_2\in\mathbb{R}_+^*$, such that
\begin{equation*}
\left\langle \lambda\nabla \cdot \boldsymbol{v}_j,\boldsymbol{u}_h\cdot \boldsymbol{n}\right\rangle_{F_j} \gtrsim \alpha_2 \left( 1-\frac{C\alpha_2}{4\epsilon}\right)\left\|\frac{\lambda^{\frac12}}{h^{\frac12}}\overline{\boldsymbol{u}}_h^j\cdot \boldsymbol{n}\right\|_{F_j}^2-\frac{C\alpha_1^2}{4\epsilon}\left\|\frac{\lambda^{\frac12}}{h^{\frac12}}\overline{\boldsymbol{u}}_h^j\cdot \boldsymbol{\tau}\right\|_{F_j}^2-2\epsilon\left\|\lambda^{\frac12}\nabla\boldsymbol{u}_h\right\|_{P_j}^2.
\end{equation*}
\end{lemma}

\begin{lemma}
\label{inequality5}
There exists $C > 0$ independent of $h$, $\mu$ and $\lambda$, but not of the mesh geometry, $\forall\boldsymbol{u}_h\in V_h^k$, on each patch $P_j$ for $\boldsymbol{v}_j\in V_h^k$ as defined in equation (\ref{defvgamma}) and $\forall \epsilon,\alpha_1,\alpha_2\in\mathbb{R}_+^*$, such that
\begin{equation*}
\left\langle 2\mu\boldsymbol{\varepsilon}\left(\boldsymbol{v}_j\right) \cdot \boldsymbol{n},\boldsymbol{u}_h \right\rangle_{F_j}\geq\alpha_2\left( 2-\frac{5C\alpha_2}{4\epsilon}\right)\left\|\frac{\mu^{\frac12}}{h^{\frac12}}\overline{\boldsymbol{u}}_h^j \cdot \boldsymbol{n}\right\|_{F_j}^2+\alpha_1\left( 1-\frac{C\alpha_1}{4\epsilon}\right)\left\|\frac{\mu^{\frac12}}{h^{\frac12}}\overline{\boldsymbol{u}}_h^j \cdot \boldsymbol{\tau}\right\|_{F_j}^2-3\epsilon\left\|\mu^{1/2}\nabla \boldsymbol{u}_h\right\|_{P_j}^2.
\end{equation*}
\end{lemma}

\begin{definition}
\label{triplenorm}
We define the triple norm of a function $\boldsymbol{w}\in V$ as
\begin{eqnarray*}
\vertiii{ \boldsymbol{w}}^2
&=&\mu\left(\left\|\nabla\boldsymbol{w}\right\|_{\Omega}^2+\left\|h^{-\frac12}\boldsymbol{w}\right\|_{\partial\Omega}^2\right)+\lambda\left(\left\|\nabla \cdot \boldsymbol{w}\right\|_{\Omega}^2+\left\|h^{-\frac12}\boldsymbol{w}\cdot\boldsymbol{n}\right\|_{\partial\Omega}^2\right).
\end{eqnarray*}
Observe that this is a norm on $V$ by the Poincar\'e inequality.
\end{definition}
\begin{lemma}
\label{lowerbound}
For $\boldsymbol{u}_h, \boldsymbol{v}_h\in V_h^k$ with $\boldsymbol{v}_h= \boldsymbol{u}_h+ \boldsymbol{v}_\Gamma$, $\boldsymbol{v}_\Gamma$ defined by equations \eqref{defvgamma} and \eqref{defvj},  there exists positive constants $\beta_0$ and $h_0$ such that the following inequality holds for $h<h_0$
$$\beta_0\vertiii{ \boldsymbol{u}_h}^2 \leq A_h(\boldsymbol{u}_h,\boldsymbol{v}_h).$$
\end{lemma}
\begin{proof}
Decomposing the bilinear form,  we can write the following
$$A_h(\boldsymbol{u}_h,\boldsymbol{v}_h)=A_h(\boldsymbol{u}_h,\boldsymbol{u}_h)+\sum_{j=1}^{N_p}A_h(\boldsymbol{u}_h,\boldsymbol{v}_j).$$
Clearly we have
\begin{equation*}
A_h(\boldsymbol{u}_h,\boldsymbol{u}_h)=2\left\|\mu^{\frac12}\boldsymbol{\varepsilon}(\boldsymbol{u}_h)\right\|_{\Omega}^2+\left\| \lambda^{\frac12}\nabla \cdot \boldsymbol{u}_h\right\|_{\Omega}^2,
\end{equation*}
and
\begin{equation*}
\begin{split}
A_h(\boldsymbol{u}_h,\boldsymbol{v}_j)=&\left(2\mu\boldsymbol{\varepsilon}(\boldsymbol{u}_h),\boldsymbol{\varepsilon}(\boldsymbol{v}_j)\right)_{P_j} -\left\langle 2\mu \boldsymbol{\varepsilon}\left(\boldsymbol{u}_h\right) \cdot \boldsymbol{n},\boldsymbol{v}_j \right\rangle_{F_j}+\left\langle 2\mu\boldsymbol{\varepsilon}\left(\boldsymbol{v}_j\right) \cdot \boldsymbol{n},\boldsymbol{u}_h \right\rangle_{F_j}\\&+\left(\lambda \nabla \cdot \boldsymbol{u}_h, \nabla \cdot \boldsymbol{v}_j \right)_{P_j}-\left\langle\lambda \nabla \cdot \boldsymbol{u}_h,\boldsymbol{v}_j \cdot \boldsymbol{n} \right\rangle_{F_j}+\left\langle\lambda \nabla \cdot \boldsymbol{v}_j,\boldsymbol{u}_h \cdot \boldsymbol{n} \right\rangle_{F_j}.
\end{split}
\end{equation*}
Using the Cauchy-Schwarz inequality and the inequalities \eqref{keyineq1} \eqref{keyineq2}, we can write the two terms defined over $P_j$ as
\begin{eqnarray*}
\left(2\mu\boldsymbol{\varepsilon}(\boldsymbol{u}_h),\boldsymbol{\varepsilon}(\boldsymbol{v}_j)\right)_{P_j}
&\geq&-\epsilon \left\|\mu^{\frac12}\boldsymbol{\varepsilon}(\boldsymbol{u}_h) \right\|_{P_j}^2 -\frac{ C\alpha_1^2}{\epsilon} \left\|\frac{\mu^{\frac12}}{h^{\frac12}}\overline{\boldsymbol{u}}_h^j \cdot \boldsymbol{\tau}\right\|_{F_j}^2-\frac{C\alpha_2^2}{\epsilon} \left\|\frac{\mu^{\frac12}}{h^{\frac12}}\overline{\boldsymbol{u}}_h^j \cdot \boldsymbol{n}\right\|_{F_j}^2,\\
\left(\lambda\nabla \cdot \boldsymbol{u}_h,\nabla \cdot \boldsymbol{v}_j \right)_{P_j}
&\geq&-\epsilon \left\|\lambda^{\frac12}\nabla \boldsymbol{u}_h\right\|_{P_j}^2-\frac{C\alpha_1^2}{4\epsilon} \left\|\frac{\lambda^{\frac12}}{h^{\frac12}}\overline{\boldsymbol{u}}_h^j \cdot \boldsymbol{\tau}\right\|_{F_j}^2-\frac{C\alpha_2^2}{4\epsilon} \left\|\frac{\lambda^{\frac12}}{h^{\frac12}}\overline{\boldsymbol{u}}_h^j \cdot \boldsymbol{n}\right\|_{F_j}^2.
\end{eqnarray*}
Combining the inequality (\ref{inequality7}) with the trace and inverse inequalities of Lemmas \ref{trace} and \ref{inverse}, followed by \eqref{keyineq1} \eqref{keyineq2} we obtain
\begin{eqnarray*}
\left\langle 2\mu \boldsymbol{\varepsilon}\left(\boldsymbol{u}_h\right) \cdot \boldsymbol{n},\boldsymbol{v}_j \right\rangle_{F_j}
&\leq&\epsilon \left\|\mu^{\frac12}\boldsymbol{\varepsilon}(\boldsymbol{u}_h) \right\|_{P_j}^2 +\frac{C\alpha_1^2}{\epsilon} \left\|\frac{\mu^{\frac12}}{h^{\frac12}}\overline{\boldsymbol{u}}_h^j \cdot \boldsymbol{\tau}\right\|_{F_j}^2+\frac{C\alpha_2^2}{\epsilon} \left\|\frac{\mu^{\frac12}}{h^{\frac12}}\overline{\boldsymbol{u}}_h^j \cdot \boldsymbol{n}\right\|_{F_j}^2,\\
\left\langle\lambda\nabla \cdot  \boldsymbol{u}_h,  \boldsymbol{v}_j \cdot \boldsymbol{n}\right\rangle_{F_j}
&\leq&\epsilon \left\|\lambda^{\frac12}\nabla\boldsymbol{u}_h\right\|_{P_j}^2+\frac{C\alpha_1^2}{4\epsilon}\left\|\frac{\lambda^{\frac12}}{h^{\frac12}}\overline{\boldsymbol{u}}_h^j \cdot  \boldsymbol{\tau}\right\|_{F_j}^2+\frac{C\alpha_2^2}{4\epsilon}\left\|\frac{\lambda^{\frac12}}{h^{\frac12}}\overline{\boldsymbol{u}}_h^j \cdot  \boldsymbol{n}\right\|_{F_j}^2.
\end{eqnarray*}
Considering Lemmas \ref{inequality6} and \ref{inequality5}  we have a lower bound for each term. Now we can write the bilinear form
\begin{equation*}
\begin{split}
A_h\left(\boldsymbol{u}_h,\boldsymbol{v}_h\right)
\geq&~2\left\|\mu^{\frac12}\boldsymbol{\varepsilon}\left(\boldsymbol{u}_h\right)\right\|_{\Omega}^2+\left\|\lambda^{\frac12}\nabla \cdot \boldsymbol{u}_h\right\|_{\Omega}^2-2\epsilon\sum_{j=1}^{N_p}\left\|\mu^{\frac12}\boldsymbol{\varepsilon}(\boldsymbol{u}_h) \right\|_{P_j}^2-\left(3\epsilon\mu+4\epsilon\lambda \right)\sum_{j=1}^{N_p}\left\|\nabla \boldsymbol{u}_h\right\|_{P_j}^2\\
&+\alpha_1\left( 1-\alpha_1\frac{9C}{4\epsilon}\right) \sum_{j=1}^{N_p}\left\|\frac{\mu^{\frac12}}{h^{\frac12}}\overline{\boldsymbol{u}}_h^j \cdot \boldsymbol{\tau}\right\|_{F_j}^2+\alpha_2\left(2-\alpha_2\frac{13C}{4\epsilon}\right) \sum_{j=1}^{N_p}\left\|\frac{\mu^{\frac12}}{h^{\frac12}}\overline{\boldsymbol{u}}_h^j \cdot \boldsymbol{n}\right\|_{F_j}^2\\
&+\alpha_1\left(-\alpha_1\frac{3C}{4\epsilon}\right) \sum_{j=1}^{N_p}\left\|\frac{\lambda^{\frac12}}{h^{\frac12}}\overline{\boldsymbol{u}}_h^j \cdot \boldsymbol{\tau}\right\|_{F_j}^2+\alpha_2\left(1 -\alpha_2\frac{3C}{4\epsilon}\right) \sum_{j=1}^{N_p}\left\|\frac{\lambda^{\frac12}}{h^{\frac12}}\overline{\boldsymbol{u}}_h^j \cdot \boldsymbol{n}\right\|_{F_j}^2.
\end{split}
\end{equation*}
The Theorem \ref{korn} gives
$$\left\|\boldsymbol{\varepsilon}\left( \boldsymbol{u}_h\right)\right\|_{\Omega}+\left|\boldsymbol{u}_h\right|_{\Gamma}\geq C_K\left\|\boldsymbol{u}_h\right\|_{H^1\left(\Omega\right)}~~~~~\forall \boldsymbol{u}_h\in V_h^k.$$
Assuming that each side $\Gamma_i$ contains at least one $F_j$, the properties of the $P_0$-projection allows us to write
\begin{equation*}
\int_{\Gamma_i}\left(P_{0}\boldsymbol{u}_h\right)^2~\text{ds}\leq\sum_{j=1}^{N_{\Gamma_i}}\int_{F_j}\left(\overline{\boldsymbol{u}}_h^j\right)^2~\text{ds},
\end{equation*}
$N_{\Gamma_i}$ is the number of $F_j$ contained in the side $\Gamma_i$. Then over all the boundaries $\Gamma_i$
$$\sum_{i=1}^{N_b}\int_{\Gamma_i}\left(P_{0}{\boldsymbol{u}}_h\right)^2~\text{ds}\leq\sum_{j=1}^{N_p}\int_{F_j}\left(\overline{\boldsymbol{u}}_h^j\right)^2~\text{ds}.$$
Then we can use the following bound
$$\left\|\boldsymbol{\varepsilon}\left( \boldsymbol{u}_h\right)\right\|_{\Omega}^2+\sum_{j=1}^{N_p}\left\|\overline{\boldsymbol{u}}_h^j\right\|_{F_j}^2\geq C_K\left\|\boldsymbol{u}_h \right\|_{H^1\left(\Omega\right)}^2~~~~~\forall \boldsymbol{u}_h\in V_h^k.$$
Using this result, we can rewrite the bilinear form $A_h\left(\boldsymbol{u}_h,\boldsymbol{v}_h\right)$ as
\begin{equation*}
\begin{split}
A_h\left(\boldsymbol{u}_h,\boldsymbol{v}_h\right)
\geq&~\left\|\lambda^{\frac12}\nabla \cdot \boldsymbol{u}_h\right\|_{\Omega}^2+2C_K\left\|\mu^{\frac12}\nabla\boldsymbol{u}_h\right\|_{\Omega\backslash P}^2+\left(2\mu C_K-5\epsilon\mu-4\epsilon\lambda\right) \sum_{j=1}^{N_p}\left\|\nabla \boldsymbol{u}_h \right\|_{P_j}^2\\
&+\left(\left(\alpha_1\left(1-\alpha_1\frac{9C}{4\epsilon}\right)-2h\right)\mu-\alpha_1^2\frac{3C}{4\epsilon}\lambda\right) \sum_{j=1}^{N_p}\left\|h^{-\frac12}\overline{\boldsymbol{u}}_h^j \cdot \boldsymbol{\tau}\right\|_{F_j}^2\\
&+\left(\left(\alpha_2\left( 2-\alpha_2\frac{13C}{4\epsilon}\right)-2h\right)\mu+\alpha_2 \left( 1-\alpha_2\frac{3C}{4\epsilon}\right)\lambda\right) \sum_{j=1}^{N_p}\left\|h^{-\frac12}\overline{\boldsymbol{u}}_h^j \cdot \boldsymbol{n}\right\|_{F_j}^2.
\end{split}
\end{equation*}
Considering the inequality (\ref{inequality4}) we obtain
\begin{equation*}
\begin{split}
A_h\left(\boldsymbol{u}_h,\boldsymbol{v}_h\right)
\geq&\left\|\lambda^{\frac12}\nabla \cdot \boldsymbol{u}_h\right\|_{\Omega}^2+2C_K\left\|\mu^{\frac12}\nabla\boldsymbol{u}_h\right\|_{\Omega\backslash P}^2+\left(C_a-C_b-C_c\right)\sum_{j=1}^{N_p}\left\|\nabla \boldsymbol{u}_h \right\|_{P_j}^2\\
&+C_b\sum_{j=1}^{N_p}\left\|h^{-\frac12}\boldsymbol{u}_h \cdot \boldsymbol{\tau}\right\|_{F_j}^2+C_c\sum_{j=1}^{N_p}\left\|h^{-\frac12}\boldsymbol{u}_h \cdot \boldsymbol{n}\right\|_{F_j}^2,
\end{split}
\end{equation*}
with the constants
\begin{eqnarray*}
C_a&=&2\mu C_K-5\epsilon\mu-4\epsilon\lambda,\\
C_b&=&\left(\alpha_1\left(1-\alpha_1\frac{9C}{4\epsilon}\right)-2h\right)\mu-\alpha_1^2\frac{3C}{4\epsilon}\lambda,\\
C_c&=&\left(\alpha_2\left( 2-\alpha_2\frac{13C}{4\epsilon}\right)-2h\right)\mu+\alpha_2 \left( 1-\alpha_2\frac{3C}{4\epsilon}\right)\lambda.
\end{eqnarray*}
First we choose $\epsilon=\frac{\mu C_K}{5\mu+4\lambda}$ so that $C_a=\mu C_K$. Fix $h<h_0$ such that $C_b$ and $C_c$ are positive respectively for
$$\frac{4\mu^2C_K}{\left(9C\mu+3C\lambda\right)\left(5\mu+4\lambda\right)}>\alpha_1~~~,~~~~~\frac{4\mu C_K\left(2\mu+\lambda\right)}{\left(13C\mu+3C\lambda\right)\left(5\mu+4\lambda\right)}>\alpha_2.$$
$C_a-C_b-C_c$ will be positive for
$$\frac{C_K}{2}>\alpha_1~~~,~~~~~\frac{\mu C_K}{2\left(2\mu+\lambda\right)}>\alpha_2.$$
By looking at the order of the constants, we can see that ${\rm O}\left(\beta_0\right)={\rm O}\left(\frac{\mu}{\lambda+\mu}\right)$ and ${\rm O}\left(h_0\right)={\rm O}\left(\frac{\mu^2}{\left(\lambda+\mu\right)^2}\right)$. If $\lambda$ is large compared to $\mu$, $h_0$ has to be very small. This reflects the locking phenomena that is well known for finite element method using low order $H^1$-conforming spaces.
\end{proof}

\begin{theorem}
\label{infsup}
There exists positive constants $\beta$ and $h_0$ such that for all functions $\boldsymbol{u}_h\in V_h^k$ and for $h<h_0$, the following inequality holds
\begin{equation*}
\beta\vertiii{\boldsymbol{u}_h}\leq\underset{\boldsymbol{v}_h\in V_h^k}{\textup{sup}} \frac{A_h\left(\boldsymbol{u}_h,\boldsymbol{v}_h \right)}{\vertiii{\boldsymbol{v}_h}}.
\end{equation*}
\end{theorem}
\begin{proof}
Considering Lemma \ref{lowerbound}, the only thing that we need to show is
\begin{equation}
\label{inequalitytriple}
\vertiii{ \boldsymbol{v}_h} \lesssim \vertiii{ \boldsymbol{u}_h}.
\end{equation}
Using the definition of the test function, the triangle inequality gives
$$\vertiii{ \boldsymbol{v}_h} \leq \vertiii{ \boldsymbol{u}_h}+\vertiii{ \boldsymbol{v}_\Gamma}.$$
The definition of the triple norm gives
\begin{equation*}
\vertiii{ \boldsymbol{v}_\Gamma}^2
=\mu\left(\left\|\nabla\boldsymbol{v}_\Gamma\right\|_{\Omega}^2+\left\|h^{-\frac12}\boldsymbol{v}_\Gamma\right\|_{\partial\Omega}^2\right)+\lambda\left(\left\|\nabla \cdot \boldsymbol{v}_\Gamma\right\|_{\Omega}^2+\left\|h^{-\frac12}\boldsymbol{v}_\Gamma\cdot \boldsymbol{n}\right\|_{\partial\Omega}^2\right).
\end{equation*}
We observe that
\begin{eqnarray*}
\alpha_1\left\|\frac{\mu^\frac12}{h^{\frac12}}\overline{\boldsymbol{u}}_h^j\cdot \boldsymbol{\tau}\right\|_{F_j}+\alpha_2\left\|\frac{\mu^\frac12}{h^{\frac12}}\overline{\boldsymbol{u}}_h^j\cdot \boldsymbol{n}\right\|_{F_j}\lesssim\alpha_1\left\|\frac{\mu^\frac12}{h^{\frac12}}\boldsymbol{u}_h\cdot \boldsymbol{\tau}\right\|_{F_j}+\alpha_2\left\|\frac{\mu^\frac12}{h^{\frac12}}\boldsymbol{u}_h\cdot \boldsymbol{n}\right\|_{F_j}&\lesssim& \vertiii{\boldsymbol{u}_h},\\
\alpha_1\left\|\frac{\lambda^\frac12}{h^{\frac12}}\overline{\boldsymbol{u}}_h^j\cdot \boldsymbol{\tau}\right\|_{F_j}+\alpha_2\left\|\frac{\lambda^\frac12}{h^{\frac12}}\overline{\boldsymbol{u}}_h^j\cdot \boldsymbol{n}\right\|_{F_j}\lesssim\alpha_1\left\|\frac{\lambda^\frac12}{h^{\frac12}}\boldsymbol{u}_h\cdot \boldsymbol{\tau}\right\|_{F_j}+\alpha_2\left\|\frac{\lambda^\frac12}{h^{\frac12}}\boldsymbol{u}_h\cdot \boldsymbol{n}\right\|_{F_j}&\lesssim& \vertiii{\boldsymbol{u}_h},
\end{eqnarray*}
using this results and recalling the inequalities \eqref{keyineq1} \eqref{keyineq2},
 it gives the appropriate upper bounds considering the definition of $\boldsymbol{v}_\Gamma$
\begin{eqnarray}
\label{ref1}
\left\| \mu^\frac12\nabla \boldsymbol{v}_\Gamma\right\|_{\Omega}&\lesssim&\vertiii{\boldsymbol{u}_h},\\
\left\| \lambda^\frac12\nabla \cdot \boldsymbol{v}_\Gamma\right\|_{\Omega}\leq\left\| \lambda^\frac12\nabla \boldsymbol{v}_\Gamma\right\|_{\Omega}&\lesssim&\vertiii{\boldsymbol{u}_h}.\nonumber
\end{eqnarray}
Using the trace inequality \ref{trace} for the boundary terms and the inequality (\ref{inequality7}) we can write
\begin{eqnarray}
\label{ref2}
\left\|\frac{\mu^\frac12}{h^{\frac12}}\boldsymbol{v}_\Gamma\right\|_{\partial\Omega}\lesssim\left\|\mu^\frac12 \nabla \boldsymbol{v}_\Gamma\right\|_{\Omega}&\lesssim&\vertiii{\boldsymbol{u}_h},\\
\left\|\frac{\lambda^\frac12}{h^{\frac12}}\boldsymbol{v}_\Gamma\cdot \boldsymbol{n}\right\|_{\partial\Omega}\lesssim\left\|\lambda^\frac12 \nabla \boldsymbol{v}_\Gamma\right\|_{\Omega}&\lesssim&\vertiii{\boldsymbol{u}_h}.\nonumber
\end{eqnarray}
We note that ${\rm O}\left(\beta\right)={\rm O}\left(\frac{\mu}{\lambda+\mu}\right)$.
\end{proof}

\subsection{A priori error estimate}
Using the stability proven in the previous section we may deduce the a priori error estimate in the triple norm. We first prove the consistency of the method in the form of a Galerkin orthogonality.
\begin{lemma}
\label{galerkin}
If $\boldsymbol{u}\in \left[H^{2}\left(\Omega\right)\right]^2 $ is the solution of (\ref{elasticity}) and $\boldsymbol{u}_h\in V_h^k$ the solution of (\ref{formulationelasticity}) the following property holds
$$A_h\left(\boldsymbol{u}-\boldsymbol{u}_h,\boldsymbol{v}_h\right)=0~~,~~~~\forall\boldsymbol{v}_h\in V_h^k.$$
\end{lemma}
\begin{proof}
We observe that $A_h\left(\boldsymbol{u},\boldsymbol{v}_h\right)=L_h\left(\boldsymbol{v}_h\right)=A_h\left(\boldsymbol{u}_h,\boldsymbol{v}_h\right),~\forall \boldsymbol{v}_h\in V_h^k.$
\end{proof}

We introduce an auxiliary norm, in order to study the a priori error estimate
$$\left\|\boldsymbol{w}\right\|_*=\vertiii{ \boldsymbol{w} }+\left\|\mu^{\frac12}h^{\frac12}\nabla \boldsymbol{w}\right\|_{\partial\Omega}+ \left\|\lambda^{\frac12}h^{\frac12}\nabla \cdot \boldsymbol{w}\right\|_{\partial\Omega}.$$
\begin{lemma}
\label{triplestar}
Let $\boldsymbol{w}\in \left[H^{2}\left(\Omega\right)\right]^2+V_h^k$ and $\boldsymbol{v}_h\in V_h^k$, there exists a positive constant $M$ such that the bilinear form $A_h\left(\cdot,\cdot\right)$ has the property
$$A_h\left(\boldsymbol{w},\boldsymbol{v}_h\right)\leq M\left\|\boldsymbol{w}\right\|_*\vertiii{ \boldsymbol{v}_h}.$$
\end{lemma}
\begin{proof}
Using the Cauchy-Schwarz inequality it is straightforward to write
\begin{eqnarray*}
\left(\lambda\nabla \cdot \boldsymbol{w},\nabla \cdot \boldsymbol{v}_h \right)_\Omega+\left( 2\mu\boldsymbol{\varepsilon}\left(\boldsymbol{w}\right),\boldsymbol{\varepsilon}\left(\boldsymbol{v}_h\right)\right)_\Omega
&\lesssim&\left\| \boldsymbol{w}\right\|_*\vertiii{ \boldsymbol{v}_h},\\
\left\langle\lambda\nabla \cdot \boldsymbol{w},  \boldsymbol{v}_h\cdot \boldsymbol{n}\right\rangle_{\partial \Omega}+\left\langle\lambda\nabla \cdot \boldsymbol{v}_h, \boldsymbol{w} \cdot \boldsymbol{n}\right\rangle_{\partial \Omega}
&\lesssim&\left\|\boldsymbol{w}\right\|_*\vertiii{ \boldsymbol{v}_h}.
\end{eqnarray*}
The trace inequality and the inequality (\ref{inequality7}) allows us to write
\begin{eqnarray*}
\left\langle 2 \mu\boldsymbol{\varepsilon}\left(\boldsymbol{w}\right) \cdot \boldsymbol{n},\boldsymbol{v}_h \right\rangle_{\partial \Omega}
&\lesssim& \left\| \mu^{\frac12} h^{\frac12}\nabla \boldsymbol{w}\right\|_{\partial\Omega}\left\|\frac{\mu^{\frac12}}{h^{\frac12}}\boldsymbol{v}_h\right\|_{\partial \Omega}\lesssim\left\| \boldsymbol{w}\right\|_*\vertiii{ \boldsymbol{v}_h},\\
\left\langle 2\mu\boldsymbol{\varepsilon}\left(\boldsymbol{v}_h\right) \cdot \boldsymbol{n},\boldsymbol{w} \right\rangle_{\partial \Omega}
&\lesssim&\left\| \mu^{\frac12} \nabla\boldsymbol{v}_h\right\|_{\Omega} \left\| \frac{\mu^{\frac12}}{h^{\frac12}} \boldsymbol{w}\right\|_{\partial\Omega}\lesssim\left\|\boldsymbol{w}\right\|_*\vertiii{ \boldsymbol{v}_h}.
\end{eqnarray*}
\end{proof}

\begin{proposition}
\label{bounderror}
If $\boldsymbol{u}\in \left[H^{k+1}\left(\Omega\right)\right]^2$ is the solution of (\ref{elasticity}) and $\boldsymbol{u}_h\in V_h^k$ the solution of (\ref{formulationelasticity}) with $h<h_0$, then there holds
$$\vertiii{\boldsymbol{u}-\boldsymbol{u}_h}\leq C_{\mu\lambda} h^k\left|\boldsymbol{u}\right|_{H^{k+1}\left(\Omega\right)},$$
where $C_{\mu\lambda}$ is a positive constant that depends on $\mu$, $\lambda$ and the mesh geometry.
\end{proposition}
\begin{proof}
Let $i_{\mathtt{SZ}}^k$ denote the Scott-Zhang interpolant \citep{Scott_1990_a}. The approximation property of the interpolant may be written for each  $K\in\mathcal{T}_h$
\begin{eqnarray*}
\left\|\boldsymbol{u}-i_{\mathtt{SZ}}^k\boldsymbol{u}\right\|_{K}+h_K\left\|\nabla\left(\boldsymbol{u}-i_{\mathtt{SZ}}^k\boldsymbol{u}\right)\right\|_{K}+h_K^2\left\|D^2\left(\boldsymbol{u}-i_{\mathtt{SZ}}^k\boldsymbol{u}\right)\right\|_{K}\lesssim h_K^{k+1}\left| \boldsymbol{u} \right|_{H^{k+1}(S_K)}.
\end{eqnarray*}
With $S_K:=\text{interior}\left(\cup\left\{\overline{K}_i|\overline{K}_i\cap\overline{K}\neq\emptyset, K_i\in\mathcal{T}_h\right\}\right)$. Using this property and the trace inequality it is straightforward to show that
\begin{eqnarray*}
\vertiii{\boldsymbol{u}-i_{\mathtt{SZ}}^k\boldsymbol{u}} &\lesssim& \left(\lambda^\frac12+\mu^\frac12\right) h^{k}\left| \boldsymbol{u} \right|_{H^{k+1}(\Omega)},\\
\left\|\boldsymbol{u}-i_{\mathtt{SZ}}^k\boldsymbol{u}\right\|_*&\lesssim& \left(\lambda^\frac12+\mu^\frac12\right) h^{k}\left| \boldsymbol{u} \right|_{H^{k+1}(\Omega)}.
\end{eqnarray*}
Using Theorem \ref{infsup}, the Galerkin orthogonality of Lemma \ref{galerkin}, and the Lemma \ref{triplestar} we deduce
$$\beta\vertiii{\boldsymbol{u}_h-i_{\mathtt{SZ}}^k\boldsymbol{u}}\leq\frac{A_h\left(\boldsymbol{u}-i_{\mathtt{SZ}}^k\boldsymbol{u},\boldsymbol{v}_h\right)}{\vertiii{\boldsymbol{v}_h}}\leq M\left\|\boldsymbol{u}-i_{\mathtt{SZ}}^k\boldsymbol{u}\right\|_*.$$
This inequality together with a triangle inequality leads to the desired estimate
$$\vertiii{\boldsymbol{u}-\boldsymbol{u}_h}\leq\vertiii{\boldsymbol{u}-i_{\mathtt{SZ}}^k\boldsymbol{u}}+\frac{M}{\beta}\left\|\boldsymbol{u}-i_{\mathtt{SZ}}^k\boldsymbol{u}\right\|_*.$$
We see that the constant in the estimate satisfies : ${\rm O}\left(C_{\mu\lambda}\right)={\rm O}\left(\beta^{-1}\left(\lambda^\frac12+\mu^{\frac12}\right)\right)$.
\end{proof}

The convergence of the $L^2$-error suffers of suboptimality of order ${\rm O}\left(h^{1/2}\right)$ due to the lack of adjoint consistency of the nonsymmetric formulation.
\begin{proposition}
\label{L2stabelast}
Let $\boldsymbol{u}\in\left[H^{k+1}(\Omega)\right]^2$ be the solution of (\ref{elasticity}) and $\boldsymbol{u}_h$ the solution of (\ref{formulationelasticity}) with $h<h_0$, then
$$\left\|\boldsymbol{u}-\boldsymbol{u}_h\right\|_{\Omega}\leq C_{\mu\lambda}' h^{k+\frac12}\left|\boldsymbol{u}\right|_{H^{k+1}(\Omega)},$$
where $C_{\mu\lambda}'$ is a positive constant that depends on $\mu$, $\lambda$ and the mesh geometry.
\end{proposition}
\begin{proof}
Let $\boldsymbol{z}$ satisfy the adjoint problem
\begin{eqnarray*}
-2\mu\nabla \cdot \boldsymbol{\varepsilon}(\boldsymbol{z})-\lambda\nabla\left( \nabla \cdot \boldsymbol{z}\right)&=& \boldsymbol{u}-\boldsymbol{u}_h ~~~~\mbox{ in } \Omega \nonumber, \\
\boldsymbol{z} &=& 0  ~~~~~~~~~~~~\mbox{ on } \partial\Omega.
\end{eqnarray*}
Then we can write
\begin{eqnarray*}
\left\|\boldsymbol{u}-\boldsymbol{u}_h\right\|_{\Omega}^2
&=&\left(\boldsymbol{u}-\boldsymbol{u}_h,-2\mu\nabla \cdot \boldsymbol{\varepsilon}(\boldsymbol{z})-\lambda\nabla\left( \nabla \cdot \boldsymbol{z}\right)\right)_{\Omega}\\
&=&\left(2\mu\boldsymbol{\varepsilon}(\boldsymbol{u}-\boldsymbol{u}_h),\boldsymbol{\varepsilon}(\boldsymbol{z})\right)_{\Omega}+\left(\lambda\nabla\cdot(\boldsymbol{u}-\boldsymbol{u}_h),\nabla\cdot\boldsymbol{z}\right)_{\Omega}\\
&&-\left\langle 2\mu(\boldsymbol{u}-\boldsymbol{u}_h),\boldsymbol{\varepsilon}(\boldsymbol{z})\cdot\boldsymbol{n}\right\rangle_{\partial \Omega}-\left\langle\lambda(\boldsymbol{u}-\boldsymbol{u}_h)\cdot\boldsymbol{n},\nabla\cdot\boldsymbol{z}\right\rangle_{\partial \Omega}\\
&=&A_h\left(\boldsymbol{u}-\boldsymbol{u}_h,\boldsymbol{z}\right)-2\left\langle 2\mu(\boldsymbol{u}-\boldsymbol{u}_h),\boldsymbol{\varepsilon}(\boldsymbol{z})\cdot\boldsymbol{n}\right\rangle_{\partial \Omega}-2\left\langle\lambda(\boldsymbol{u}-\boldsymbol{u}_h)\cdot\boldsymbol{n},\nabla\cdot\boldsymbol{z}\right\rangle_{\partial \Omega}.
\end{eqnarray*}
By Lemma $\ref{galerkin}$, using $(\boldsymbol{z}-i_{\mathtt{SZ}}^1\boldsymbol{z})|_{\partial\Omega}\equiv 0$ and similar arguments as in the proof of Lemma \ref{triplestar} we deduce that
\begin{eqnarray}
A_h(\boldsymbol{u}-\boldsymbol{u}_h,\boldsymbol{z})
&=&A_h\left(\boldsymbol{u}-\boldsymbol{u}_h,\boldsymbol{z}-i_{\mathtt{SZ}}^1\boldsymbol{z}\right)\nonumber\\
&=&\left(2\mu\boldsymbol{\varepsilon}(\boldsymbol{u}-\boldsymbol{u}_h),\boldsymbol{\varepsilon}(\boldsymbol{z}-i_{\mathtt{SZ}}^1\boldsymbol{z})\right)_{\Omega}+\left(\lambda\nabla\cdot(\boldsymbol{u}-\boldsymbol{u}_h),\nabla\cdot(\boldsymbol{z}-i_{\mathtt{SZ}}^1\boldsymbol{z})\right)_{\Omega}\nonumber\\
&&+\left\langle 2\mu(\boldsymbol{u}-\boldsymbol{u}_h),\boldsymbol{\varepsilon}(\boldsymbol{z}-i_{\mathtt{SZ}}^1\boldsymbol{z})\cdot\boldsymbol{n}\right\rangle_{\partial \Omega}+\left\langle\lambda(\boldsymbol{u}-\boldsymbol{u}_h)\cdot\boldsymbol{n},\nabla\cdot(\boldsymbol{z}-i_{\mathtt{SZ}}^1\boldsymbol{z})\right\rangle_{\partial \Omega}\nonumber\\
&\lesssim&\vertiii{\boldsymbol{u}-\boldsymbol{u}_h}\left\|\boldsymbol{z}-i_{\mathtt{SZ}}^1\boldsymbol{z}\right\|_*\nonumber\\
&\lesssim&\left(\lambda^\frac12+\mu^{\frac12}\right)h\vertiii{\boldsymbol{u}-\boldsymbol{u}_h} \left|\boldsymbol{z}\right|_{H^2(\Omega)}.
\label{L2proofeq1}
\end{eqnarray}
The global trace inequalities $\left\|\boldsymbol{\varepsilon}(\boldsymbol{z})\cdot\boldsymbol{n}\right\|_{\partial\Omega}\lesssim\left\|\boldsymbol{z}\right\|_{H^2(\Omega)}$ and $\left\|\nabla\cdot\boldsymbol{z}\right\|_{\partial\Omega}\lesssim\left\|\boldsymbol{z}\right\|_{H^2(\Omega)},$
lead to
\begin{equation}
\left|\left\langle 2\mu(\boldsymbol{u}-\boldsymbol{u}_h),\boldsymbol{\varepsilon}(\boldsymbol{z})\cdot\boldsymbol{n}\right\rangle_{\partial \Omega}\right|+\left|\left\langle\lambda(\boldsymbol{u}-\boldsymbol{u}_h)\cdot\boldsymbol{n},\nabla\cdot\boldsymbol{z}\right\rangle_{\partial \Omega}\right|\lesssim \left(\lambda^\frac12+\mu^{\frac12}\right)h^{\frac12}\vertiii{\boldsymbol{u}-\boldsymbol{u}_h}\left\|\boldsymbol{z}\right\|_{H^2(\Omega)}.
\label{L2proofeq2}
\end{equation}
Using inequalities (\ref{L2proofeq1}) and (\ref{L2proofeq2}) we obtain
$$\left\|\boldsymbol{u}-\boldsymbol{u}_h\right\|_{\Omega}^2\lesssim C_{\mu\lambda}\left(\lambda^\frac12+\mu^{\frac12}\right)\left(h+h^{\frac12}\right)h^k\left|\boldsymbol{u}\right|_{H^{k+1(\Omega)}}\left\|\boldsymbol{z}\right\|_{H^2(\Omega)}.$$
We conclude applying the regularity estimate $\left\|\boldsymbol{z}\right\|_{H^2(\Omega)}\lesssim\left\|\boldsymbol{u}-\boldsymbol{u}_h\right\|_{\Omega}$. ${\rm O}\left(C_{\mu\lambda}'\right)={\rm O}\left(C_{\mu\lambda}\left(\lambda^\frac12+\mu^{\frac12}\right)\right)$.
\end{proof}

\section{Incompressible elasticity}
In this part we consider the problem (\ref{stokes}) and we prove the stability for this configuration similarly as in the previous part for the compressible case. For incompressible elasticity we have to manage one more unknown, the pressure. We choose to work with equal order interpolation for the velocity and the pressure and add a pressure stabilization to recover stability. Note that in this part we re-define the bilinear forms, the triple norm and the star norm. We have the following weak formulation: find $\left(\boldsymbol{u},p\right)\in V_g\times Q$ such that
$$a\left[\left(\boldsymbol{u},p\right),\left(\boldsymbol{v},q\right)\right]=\left(\boldsymbol{f},\boldsymbol{v}\right)_\Omega~~~~~~~~~\forall \left(\boldsymbol{v},q\right)\in V_0\times Q,$$
with
$$a\left[\left(\boldsymbol{u},p\right),\left(\boldsymbol{v},q\right)\right]=\left(2\mu\boldsymbol{\varepsilon}(\boldsymbol{u}),\boldsymbol{\varepsilon}(\boldsymbol{v})\right)_\Omega-\left(p,\nabla\cdot\boldsymbol{v}\right)_\Omega+\left(\nabla\cdot\boldsymbol{u},q\right)_\Omega.$$
\subsection{Finite element formulation}
The nonsymmetric Nitsche's method applied to the incompressible elasticity (\ref{stokes}) gives the following variational formulation, find $\boldsymbol{u}_h\in V_h^k$ and $p_h\in Q_h^k$ such that
\begin{equation}
\label{formulationstokes}
A_h\left[\left(\boldsymbol{u}_h,p_h\right),\left(\boldsymbol{v}_h,q_h\right)\right]=L_h\left(\boldsymbol{v}_h,q_h\right)~~~~~~~~~\forall\left(\boldsymbol{v}_h,q_h\right)\in V_h^k\times Q_h^k,
\end{equation}
where the bilinear forms $A_h$ and $L_h$ are defined as
\begin{equation*}
\begin{split}
A_h\left[\left(\boldsymbol{u}_h,p_h\right),\left(\boldsymbol{v}_h,q_h\right)\right]=&~a\left[\left(\boldsymbol{u}_h,p_h\right),\left(\boldsymbol{v}_h,q_h\right)\right]-b\left(\boldsymbol{u}_h,\boldsymbol{v}_h,p_h\right)+b\left(\boldsymbol{v}_h,\boldsymbol{u}_h,q_h\right)+S_h\left(\boldsymbol{u}_h,p_h,q_h\right),\\
L_h\left(\boldsymbol{v}_h,q_h\right)=&~\left(\boldsymbol{f},\boldsymbol{v}_h+\frac{\gamma}{\mu}h^2\nabla q_h\right)_{\Omega}+b\left(\boldsymbol{v}_h,\boldsymbol{g},q_h\right).
\end{split}
\end{equation*}
The bilinear form $b$ is defined as
\begin{eqnarray*}
b\left(\boldsymbol{u}_h,\boldsymbol{v}_h,p_h\right)&=&\left\langle\left(2\mu\boldsymbol{\varepsilon}\left(\boldsymbol{u}_h\right)-p_h\mathbb{I}_{2\times 2}\right)\cdot \boldsymbol{n},\boldsymbol{v}_h\right\rangle_{\partial\Omega}.
\end{eqnarray*}
$S_h$ denotes the stabilization term, we define
\begin{equation*}
S_h\left(\boldsymbol{u}_h,p_h,q_h\right)=\frac{\gamma}{\mu}\sum_{K\in\mathcal{T}_h}\int_K h^2\left(-2\mu\nabla \cdot \boldsymbol{\varepsilon} \left(\boldsymbol{u}_h\right)+\nabla p_h\right)\nabla q_h~\text{d}x,
\end{equation*}
% modify review 1
this term is necessary as we want to use equal order interpolation.
\subsection{Stability}
We proceed similarly as for the compressible case, we first define the triple norm.
\begin{definition}
\label{triplenorm2}
We define the triple norm of $\left(\boldsymbol{w},\varrho\right)\in V\times L^2\left(\Omega\right)$ as
\begin{equation*}
\vertiii{ \left(\boldsymbol{w},\varrho\right)}^2=~
\mu\left(\left\|\nabla \boldsymbol{w}\right\|_{\Omega}^2+\left\|h^{-\frac12}\boldsymbol{w}\right\|_{\partial\Omega}^2\right)+\frac{1}{\mu}\left\|h\nabla \varrho \right\|_{\Omega}^2.
\end{equation*}
\end{definition}
\begin{lemma}
\label{lowerbound2}
For $\boldsymbol{u}_h,\boldsymbol{v}_h\in V_h^k$ with $\boldsymbol{v}_h= \boldsymbol{u}_h+ \boldsymbol{v}_\Gamma$, $\boldsymbol{v}_\Gamma$ defined by equations \eqref{defvgamma} \eqref{defvj}, and $q_h=p_h$, there exists positive constants $\beta_0$ and $h_0$ such that the following inequality holds for $h<h_0$
$$\beta_0\vertiii{ \left(\boldsymbol{u}_h,p_h\right)}^2 \leq A_h\left[\left(\boldsymbol{u}_h,p_h\right),\left(\boldsymbol{v}_h,q_h\right)\right].$$
\end{lemma}
\begin{proof}
Decomposing the bilinear form,  we can write the following
\begin{equation*}
A_h\left[\left(\boldsymbol{u}_h,p_h\right),\left(\boldsymbol{v}_h,q_h\right)\right]=A_h\left[\left(\boldsymbol{u}_h,p_h\right),\left(\boldsymbol{u}_h,p_h\right)\right]+\sum_{j=1}^{N_p}A_h\left[\left(\boldsymbol{u}_h,p_h\right),\left(\boldsymbol{v}_j,0\right)\right].
\end{equation*}
Using the Cauchy-Schwarz inequality and an inverse inequality we can write
\begin{eqnarray*}
A_h\left[\left(\boldsymbol{u}_h,p_h\right),\left(\boldsymbol{u}_h,p_h\right)\right]
&\geq&2\left\|\mu^{\frac12}\boldsymbol{\varepsilon}\left( \boldsymbol{u}_h\right)\right\|_{\Omega}^2-\frac{\gamma}{\mu}\left\|2h\mu\nabla \cdot \boldsymbol{\varepsilon}\left( \boldsymbol{u}_h\right)\right\|_\Omega \left\|h\nabla p_h\right\|_\Omega+\frac{\gamma}{\mu}\left\|h\nabla p_h\right\|_\Omega^2\\
&\geq& 2\left(1-\epsilon'\right)\left\|\mu^{\frac12}\boldsymbol{\varepsilon}\left( \boldsymbol{u}_h\right)\right\|_{\Omega}^2+\frac\gamma\mu\left(1-\frac{C\gamma}{4\epsilon'}\right)\left\|h\nabla p_h\right\|_\Omega^2.
\end{eqnarray*}
The second part can be written as
\begin{equation*}
A_h\left[\left(\boldsymbol{u}_h,p_h\right),\left(\boldsymbol{v}_j,0\right)\right]=\left(2\mu\boldsymbol{\varepsilon}\left(\boldsymbol{u}_h\right),\boldsymbol{\varepsilon}\left(\boldsymbol{v}_j\right)\right)_{P_j} +\left(\nabla p_h, \boldsymbol{v}_j\right)_{P_j}-\left\langle 2\mu \boldsymbol{\varepsilon}\left(\boldsymbol{u}_h\right) \cdot \boldsymbol{n},\boldsymbol{v}_j \right\rangle_{F_j}+\left\langle 2\mu\boldsymbol{\varepsilon}\left(\boldsymbol{v}_j\right) \cdot \boldsymbol{n},\boldsymbol{u}_h \right\rangle_{F_j}.
\end{equation*}
Term by term we can obtain a lower bound of each term, note that most of the terms have been studied in the compressible case. The lower bound of the only remaining term can be found using the inequalities \eqref{keyineq1} \eqref{keyineq2} and the inequality (\ref{inequality7}), we get
$$\left(\nabla p_h,\boldsymbol{v}_j\right)_{P_j}\geq-\frac{\epsilon}{\mu}\left\|h \nabla p_h\right\|_{P_j}^2-\frac{C\alpha_1^2\mu}{2 \epsilon}\left\|h^{-\frac12}\overline{\boldsymbol{u}}_h^j \cdot \boldsymbol{\tau}\right\|_{F_j}^2-\frac{C \alpha_2^2\mu}{2 \epsilon}\left\|h^{-\frac12}\overline{\boldsymbol{u}}_h^j \cdot \boldsymbol{n}\right\|_{F_j}^2.$$
The full bilinear form gives
\begin{equation*}
\begin{split}
A_h\left[\left(\boldsymbol{u}_h,p_h\right),\left(\boldsymbol{v}_h,q_h\right)\right]
\geq&~2\left(1-\epsilon'\right)\left\|\mu^{\frac12}\boldsymbol{\varepsilon}\left( \boldsymbol{u}_h\right)\right\|_{\Omega}^2+\frac\gamma\mu\left(1-\frac{C\gamma}{4\epsilon'}\right)\left\|h\nabla p_h \right\|_{\Omega}^2\\
&-2\epsilon\sum_{j=1}^{N_p}\left\| \mu^{\frac12} \boldsymbol{\varepsilon}\left(\boldsymbol{u}_h\right)\right\|_{P_j}^2-\frac{\epsilon}{\mu}\sum_{j=1}^{N_p}\left\|  h\nabla p_h\right\|_{P_j}^2-3\epsilon\sum_{j=1}^{N_p}\left\| \mu^{\frac12} \nabla \boldsymbol{u}_h\right\|_{P_j}^2\\
&-\frac{C\alpha_1^2\mu}{2\epsilon}\sum_{j=1}^{N_p}\left\|h^{-\frac12}\overline{\boldsymbol{u}}_h^j \cdot \boldsymbol{\tau}\right\|_{F_j}^2-\frac{C\alpha_2^2\mu}{2\epsilon}\sum_{j=1}^{N_p}\left\|h^{-\frac12}\overline{\boldsymbol{u}}_h^j \cdot \boldsymbol{n}\right\|_{F_j}^2\\
&+\alpha_1\left( 1-\alpha_1\frac{11C}{4\epsilon}\right)\sum_{j=1}^{N_p}\left\|\frac{\mu^{\frac12}}{h^{\frac12}}\overline{\boldsymbol{u}}_h^j \cdot \boldsymbol{\tau}\right\|_{F_j}^2+\alpha_2\left( 2-\alpha_2\frac{15C}{4\epsilon}\right)\sum_{j=1}^{N_p}\left\|\frac{\mu^{\frac12}}{h^{\frac12}}\overline{\boldsymbol{u}}_h^j \cdot \boldsymbol{n}\right\|_{F_j}^2.
\end{split}
\end{equation*}
Similarly as for the compressible case, using the Theorem \ref{korn} and the inequality (\ref{inequality4}) we obtain 
\begin{equation*}
\begin{split}
A_h\left[\left(\boldsymbol{u}_h,p_h\right),\left(\boldsymbol{v}_h,q_h\right)\right]\geq
&~C_a\left\|\mu^{\frac12}\nabla \boldsymbol{u}_h\right\|_{\Omega\backslash P}^2+C_b\left\|h\nabla p_h \right\|_{\Omega\backslash P}^2+\left(C_c-C_e-C_f\right)\sum_{j=1}^{N_p}\left\| \mu^\frac12\nabla\boldsymbol{u}_h\right\|_{P_j}^2\\&+C_d\sum_{j=1}^{N_p}\left\|  h\nabla p_h\right\|_{P_j}^2+C_e\sum_{j=1}^{N_p}\left\|\frac{\mu^{\frac12}}{h^{\frac12}}\boldsymbol{u}_h \cdot \boldsymbol{\tau}\right\|_{F_j}^2+C_f\sum_{j=1}^{N_p}\left\|\frac{\mu^{\frac12}}{h^{\frac12}}\boldsymbol{u}_h \cdot \boldsymbol{n}\right\|_{F_j}^2,
\end{split}
\end{equation*}
with the constants
\begin{eqnarray*}
C_a&=&2C_K\left(1-\epsilon'\right),\\
C_b&=&\frac{\gamma}{\mu}\left(1-\frac{C\gamma}{4\epsilon'}\right),\\
C_c&=&2C_K\left(1-\epsilon'\right)-5\epsilon,\\
C_d&=&\frac{\gamma}{\mu}\left(1-\frac{C\gamma}{4\epsilon'}\right)-\frac{\epsilon}{\mu},\\
C_e&=&\alpha_1\left( 1-\alpha_1\frac{11C}{4\epsilon}\right)-2h\left(1-\epsilon'\right),\\
C_f&=&\alpha_2\left( 2-\alpha_2\frac{15C}{4\epsilon}\right)-2h\left(1-\epsilon'\right).
\end{eqnarray*}
We choose $\epsilon=\frac{\gamma^2}{4}$ and $\epsilon'=\frac{1}{4}$. Taking $\gamma<\frac{1}{C+\frac{1}{4}}$, for $h$ small enough $C_e$ and $C_f$ will be positive respectively for
$$\frac{\gamma^2}{11C}>\alpha_1~~~,~~~~~~~\frac{2\gamma^2}{15C}>\alpha_2.$$
$C_c-C_e-C_f$ will be positive for
$$\sqrt{\frac{2C_K}{5}}>\gamma~~~,~~~~~~~\frac{C_K}{2}>\alpha_1~~~,~~~~~~~\frac{C_K}{4}>\alpha_2.$$
$h_0$ is the biggest value of $h$ that can be considered, we observe that ${\rm O}\left(\beta_0\right)={\rm O}\left(1\right)$, ${\rm O}\left(h_0\right)={\rm O}\left(1\right)$.
\end{proof}

We remark that contrary to the case of compressible elasticity we see that the conditions on the constants are independent of the physical parameters, this reflects that the mixed method is locking free.
\begin{theorem}
\label{infsup2}
There exists positive constants $\beta$ and $h_0$ such that for all functions $\left(\boldsymbol{u}_h,p_h\right)\in V_h^k\times Q_h^k$ and for $h<h_0$, the following inequality holds
\begin{equation*}
\beta\vertiii{ \left(\boldsymbol{u}_h,p_h\right)}\leq\underset{\left(\boldsymbol{v}_h,q_h\right)\in V_h^k\times Q_h^k}{\textup{sup}} \frac{A_h\left[\left(\boldsymbol{u}_h,p_h\right),\left(\boldsymbol{v}_h,q_h\right)\right]}{\vertiii{ \left(\boldsymbol{v}_h,q_h\right)}}.
\end{equation*}
\end{theorem}
\begin{proof}
Considering Lemma \ref{lowerbound2}, the only thing that we need to show is
\begin{equation*}
\vertiii{ \left(\boldsymbol{v}_h,q_h\right)} \lesssim \vertiii{ \left(\boldsymbol{u}_h,p_h\right)}.
\end{equation*}
Using the definition of the test functions, the triangle inequality gives
$$\vertiii{ \left(\boldsymbol{v}_h,q_h\right)} \leq \vertiii{ \left(\boldsymbol{u}_h,p_h\right)}+\vertiii{ \left(\boldsymbol{v}_\Gamma,0\right)}.$$
The triple norm of $(\boldsymbol{v}_\Gamma,0)$ is
\begin{equation*}
\vertiii{ \left(\boldsymbol{v}_\Gamma,0\right)}^2=~
\mu\left(\left\|\nabla \boldsymbol{v}_\Gamma\right\|_{\Omega}^2+\left\|h^{-\frac12}\boldsymbol{v}_\Gamma\right\|_{\partial\Omega}^2\right).
\end{equation*}
The claim follows from equations (\ref{ref1}, \ref{ref2}) of Theorem \ref{infsup}.
Note that ${\rm O}\left(\beta\right)={\rm O}\left(1\right)$.
\end{proof}

\subsection{A priori error estimate}
The stability proven in the previous section leads to the study of the error estimate in the triple norm, the Galerkin orthogonality is characterized by the following consistency relation.
\begin{lemma}
\label{galerkin2}
If $\left(\boldsymbol{u},p\right)\in \left[H^{2}\left(\Omega\right)\right]^2\times H^1\left(\Omega\right)$ is the solution of (\ref{stokes}) and $\left(\boldsymbol{u}_h,p_h\right)\in V_h^k\times Q_h^k$ the solution of (\ref{formulationstokes}) the the following property holds
$$A_h\left[\left(\boldsymbol{u}-\boldsymbol{u}_h,p-p_h\right),\left(\boldsymbol{v}_h,q_h\right)\right]=0.$$
\end{lemma}
The star norm of $\left( \boldsymbol{w},\varrho\right)$ used for the continuity of $A_h\left[\left(\cdot,\cdot\right),\left(\cdot,\cdot\right)\right]$ is defined as
\begin{equation*}
\begin{split}
\left\|\left( \boldsymbol{w},\varrho\right) \right\|_*:=\vertiii{ \left(  \boldsymbol{w},\varrho\right) }+\left\|\mu^{\frac12}h^{\frac12}\nabla \boldsymbol{w} \right\|_{\partial\Omega}+\left\| \varrho\right\|_{\Omega}&+\left\| h^{\frac12}\varrho\right\|_{\partial\Omega}+\left\|h^{-1}\boldsymbol{w}\right\|_{\Omega}\\
&+\left(\sum_{K\in\mathcal{T}_h}\left\|h\mu^{\frac12}\nabla\cdot\boldsymbol{\varepsilon}\left(\boldsymbol{w}\right)\right\|_{K}^2\right)^\frac12.
\end{split}
\end{equation*}
\begin{lemma}
\label{triplestar2}
Let $\left(\boldsymbol{w},\varrho\right)\in\left(\left[H^{2}\left(\Omega\right)\right]^2+V_h^{k}\right)\times\left(H^1\left(\Omega\right)+Q_h^k\right)$ and $\left(\boldsymbol{v}_h,q_h\right)\in V_h^k\times Q_h^k$ there exists a positive constant $M$ such that the bilinear form $A_h\left[\left(\cdot,\cdot\right),\left(\cdot,\cdot\right)\right]$ has the property
$$A_h\left[\left(\boldsymbol{w},\varrho\right),\left(\boldsymbol{v}_h,q_h\right)\right]\leq M\left\|\left( \boldsymbol{w},\varrho\right) \right\|_*\vertiii{ \left(\boldsymbol{v}_h,q_h\right)}.$$
\end{lemma}
\begin{proof}
The proof of the Lemma \ref{triplestar} gives us the desired upper bound for most of the terms. The integration by parts gives
$$\left(\nabla \varrho, \boldsymbol{v}_h\right)_{\Omega}=\left\langle \varrho\cdot  \boldsymbol{n}, \boldsymbol{v}_h\right\rangle_{\partial\Omega}-\left(\varrho,\nabla \cdot  \boldsymbol{v}_h\right)_{\Omega}.$$
Using the Cauchy-Schwarz inequality we obtain
\begin{eqnarray*}
\left\langle \varrho\cdot \boldsymbol{n},\boldsymbol{v}_h \right\rangle_{\partial \Omega}-\left(\varrho,\nabla \cdot \boldsymbol{v}_h\right)_{\Omega}-\left(\nabla q_h,\boldsymbol{w}\right)_{\Omega}
&\lesssim&\left\| \left( \boldsymbol{w},\varrho\right)\right\|_*\vertiii{ \left(\boldsymbol{v}_h,q_h\right)},\\
\sum_{K\in\mathcal{T}_h}\left(h^2\left(-2\mu\nabla \cdot \boldsymbol{\varepsilon}\left( \boldsymbol{w}\right)+ \nabla \varrho \right),\nabla q_h\right)_K
&\lesssim&\left\| \left( \boldsymbol{w},\varrho\right)\right\|_*\vertiii{ \left(\boldsymbol{v}_h,q_h\right)}.
\end{eqnarray*}
Note that the second line corresponds to the stabilization term.
\end{proof}

\begin{proposition}
\label{bounderror2}
If $\left(\boldsymbol{u},p\right)\in\left[H^{k+1}\left(\Omega\right)\right]^2\times H^k\left(\Omega\right)$ is the solution of (\ref{stokes}) and $\left(\boldsymbol{u}_h,p_h\right)$ the solution of (\ref{formulationstokes}) with $h<h_0$, then there holds
$$\vertiii{ \left(\boldsymbol{u}-\boldsymbol{u}_h,p-p_h\right)}\leq h^k\left(C_{u\mu}\left|\boldsymbol{u}\right|_{H^{k+1}\left(\Omega\right)}+C_{p\mu}\left|p\right|_{H^k\left(\Omega\right)}\right).$$
where $C_{u\mu}$ and $C_{p\mu}$ are positive constants that depends on $\mu$ and the mesh geometry.
\end{proposition}
\begin{proof}
Let $i_{\mathtt{SZ}}^k$ denote the Scott-Zhang interpolant \citep{Scott_1990_a}, the approximation properties for each $K\in\mathcal{T}_h$ gives
\begin{eqnarray*}
\left\|\boldsymbol{u}-i_{\mathtt{SZ}}^k\boldsymbol{u}\right\|_{K}+h_K\left\|\nabla\left(\boldsymbol{u}-i_{\mathtt{SZ}}^k\boldsymbol{u}\right)\right\|_{K}+h_K^2\left\|D^2\left(\boldsymbol{u}-i_{\mathtt{SZ}}^k\boldsymbol{u}\right)\right\|_{K}&\lesssim& h_K^{k+1}\left| \boldsymbol{u} \right|_{H^{k+1}({S_K})},\\
\left\|p-i_{\mathtt{SZ}}^kp\right\|_{K}+h_K\left\|\nabla\left(p-i_{\mathtt{SZ}}^kp\right)\right\|_{K}&\lesssim& h_K\left|p\right|_{H^{k}({S_K})}.
\end{eqnarray*}
Using these properties and the trace inequality, it is straightforward to show that
\begin{eqnarray*}
\vertiii{ \left(\boldsymbol{u}-\boldsymbol{u}_h,p-p_h\right)}&\lesssim& h^k\left(\mu^{\frac12}\left|\boldsymbol{u}\right|_{H^{k+1}\left(\Omega\right)}+\mu^{-\frac12}\left|p\right|_{H^k\left(\Omega\right)}\right),\\
\left\| \left(\boldsymbol{u}-\boldsymbol{u}_h,p-p_h\right)\right\|_*&\lesssim& h^k\left(\mu^{\frac12}\left|\boldsymbol{u}\right|_{H^{k+1}\left(\Omega\right)}+\mu^{-\frac12}\left|p\right|_{H^k\left(\Omega\right)}\right).
\end{eqnarray*}
Using Theorem \ref{infsup2}, Galerkin orthogonality and the Lemma \ref{triplestar2} we obtain
$$\beta\vertiii{ \left(\boldsymbol{u}_h-i_{\mathtt{SZ}}^k\boldsymbol{u},p_h-i_{\mathtt{SZ}}^kp\right)}\leq\frac{A_h\left[\left(\boldsymbol{u}_h-i_{\mathtt{SZ}}^k\boldsymbol{u},p_h-i_{\mathtt{SZ}}^kp\right),\left(\boldsymbol{v}_h,q_h\right)\right]}{\vertiii{ \left(\boldsymbol{v}_h,q_h\right)}}\leq M\left\| \left(\boldsymbol{u}-i_{\mathtt{SZ}}^k\boldsymbol{u},p-i_{\mathtt{\mathtt{SZ}}}^k p\right)\right\|_*.$$
Using this property and the triangle inequality we can write
$$\vertiii{ \left(\boldsymbol{u}-\boldsymbol{u}_h,p-p_h\right)}\leq\vertiii{ \left(\boldsymbol{u}-i_{\mathtt{SZ}}^k\boldsymbol{u},p-i_{\mathtt{SZ}}^kp\right)}+\frac{M}{\beta}\left\| \left(\boldsymbol{u}-i_{\mathtt{SZ}}^k\boldsymbol{u},p-i_{\mathtt{SZ}}^kp\right)\right\|_*.$$
We note that ${\rm O}\left(C_{u\mu}\right)={\rm O}\left(\mu^{\frac12}\right)$ and ${\rm O}\left(C_{p\mu}\right)={\rm O}\left(\mu^{-\frac12}\right)$.
\end{proof}

The convergence of the $L^2$-error of the velocities with the order ${\rm O}\left(h^{k+\frac12}\right)$ may be proven similarly as in Proposition \ref{L2stabelast}.
\begin{proposition}
Let $\left(\boldsymbol{u},p\right)\in \left[H^{k+1}\left(\Omega\right)\right]^2\times H^k\left(\Omega\right)$ be the solution of (\ref{stokes}) and $\left(\boldsymbol{u}_h,p_h\right)\in V_h^k\times Q_h^k$ the solution of (\ref{formulationstokes}) with $h<h_0$, then
$$\left\|p-p_h\right\|_{\Omega}\leq h^k\left(C_{u\mu}'\left|u\right|_{H^{k+1}\left(\Omega\right)}+C_{p\mu}'\left|p\right|_{H^{k}\left(\Omega\right)}\right),$$
where $C_{u\mu}'$ and $C_{p\mu}'$ are positive constants that depends on $\mu$ and the mesh geometry.
\end{proposition}
\begin{proof}
By the surjectivity of the divergence operator  $\nabla\cdot:H_0^1\left(\Omega\right)\rightarrow L^2_0\left(\Omega\right)$ \citep[see,][]{Girault_1986_a}, there exists $\boldsymbol{v}_p\in V_0$ such that $\nabla \cdot \boldsymbol{v}_p=p-p_h$.
Therefore we may write (using the Lemma \ref{galerkin2} and observing that $\left(\boldsymbol{v}_p-i_{\mathtt{SZ}}\boldsymbol{v}_p\right)|_{\partial\Omega}=0$)
\begin{eqnarray*}
\left\|p-p_h\right\|_{\Omega}^2
&=&\left(p-p_h,\nabla\cdot\boldsymbol{v}_p\right)+A_h\left[\left(\boldsymbol{u}-\boldsymbol{u}_h,p-p_h\right),\left(i_\mathtt{SZ}\boldsymbol{v}_p,0\right)\right]\\
&=&\left(p-p_h,\nabla\cdot\left(\boldsymbol{v}_p-i_{\mathtt{SZ}}\boldsymbol{v}_p\right)\right)_{\Omega}\\
&&+\left(2\mu\boldsymbol{\varepsilon}\left(\boldsymbol{u}-\boldsymbol{u}_h\right),\boldsymbol{\varepsilon}\left(i_{\mathtt{SZ}} \boldsymbol{v}_p\right)\right)_\Omega+\left\langle 2\mu\boldsymbol{\varepsilon}\left(i_{\mathtt{SZ}} \boldsymbol{v}_p\right)\cdot\boldsymbol{n}, \boldsymbol{u}-\boldsymbol{u}_h\right\rangle_{\partial\Omega}\\
&=&-\left(\nabla\left(p-p_h\right),\boldsymbol{v}_p-i_{\mathtt{SZ}}\boldsymbol{v}_p\right)_{\Omega}\\
&&+\left(2\mu\boldsymbol{\varepsilon}\left(\boldsymbol{u}-\boldsymbol{u}_h\right),\boldsymbol{\varepsilon}\left(i_{\mathtt{SZ}} \boldsymbol{v}_p\right)\right)_\Omega+\left\langle 2\mu\boldsymbol{\varepsilon}\left(i_{\mathtt{SZ}} \boldsymbol{v}_p\right)\cdot\boldsymbol{n}, \boldsymbol{u}-\boldsymbol{u}_h\right\rangle_{\partial\Omega}\\
&\lesssim&\frac{1}{\mu^\frac12}\left\|h\nabla\left(p-p_h\right)\right\|_\Omega h^{-1}\left\|\mu^\frac12\left(\boldsymbol{v}_p-i_{\mathtt{SZ}}\boldsymbol{v}_p\right)\right\|_{\Omega}\\
&&+2\left\|\mu^\frac12\nabla\left(\boldsymbol{u}-\boldsymbol{u}_h\right)\right\|_{\Omega}\left\|\mu^{\frac12}\nabla i_{\mathtt{SZ}}\boldsymbol{v}_p\right\|_{\Omega}+\left\|\mu^{\frac12}\nabla i_{\mathtt{SZ}}\boldsymbol{v}_p\right\|_{\Omega}\left\|\frac{\mu^{\frac12}}{h^\frac12}\left(\boldsymbol{u}-\boldsymbol{u}_h\right)\right\|_{\Omega}\\
&\lesssim&\mu^\frac12\vertiii{\left(\boldsymbol{u}-\boldsymbol{u}_h\right),\left(p-p_h\right)}\left|\boldsymbol{v}_p\right|_{H^1\left(\Omega\right)}.
\end{eqnarray*}
We conclude by applying the stability $\left\|\boldsymbol{v}_p\right\|_{H^1\left(\Omega\right)}\leq C_{\boldsymbol{v}_p}\left\|p-p_h\right\|_\Omega$. We observe that ${\rm O}\left(C_{u\mu}'\right)={\rm O}\left(\mu\right)$ and ${\rm O}\left(C_{p\mu}'\right)={\rm O}\left(1\right)$.
\end{proof}

\section{Numerical results}
In this section we will present some numerical experiments verifying the above theory. The package FreeFem++ \citep{Hecht_2012_a} was used for the numerical study. In the first two sections we consider the domain $\Omega$ as the unit square $\left[0,1\right]\times\left[0,1\right]$. For compressible and incompressible elasticity we use a manufactured solution to test the precision of the method. In the third section we study the performance of the penalty free Nitsche's method for the Cook's membrane problem.

\subsection{Compressible elasticity}
The two dimensional function below is a manufactured solution considered for the tests
\begin{equation*}
\boldsymbol{u}=\begin{pmatrix}\left(x^5-x^4\right)\left(y^3-y^2\right)\\ \left(x^4-x^3\right)\left(y^6-y^5\right)\end{pmatrix}.
\end{equation*}
The nonsymmetric Nitsche's method given by equation (\ref{formulationelasticity}) is used to compute approximations on a series of structured meshes. We consider first and second order polynomials and we study the convergence rates of the error in the $H^1$- and $L^2$-norms. We choose $\mu=1$ and consider several values of $\lambda$ in order to see numerically the locking phenomena for large values of $\lambda$ compared to $\mu$.
\begin{figure}[h!]
\begin{center}
\includegraphics[scale=0.27]{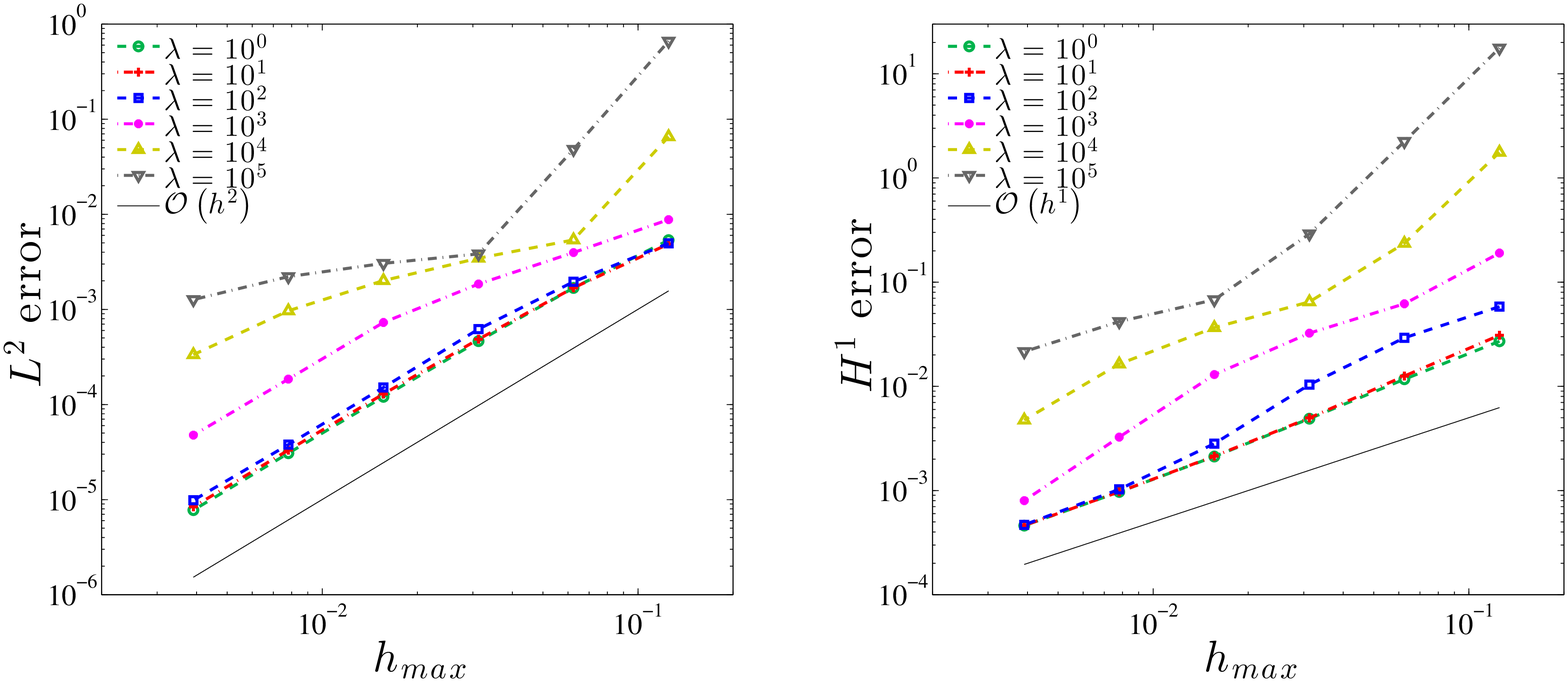}
\caption{Compressible elasticity, $V_h^1$: error versus the maximal element diameter $h_{max}$. Left: $L^2$-error, right: $H^1$-error.}
\label{linear_elasticity_graph_D1}
\end{center}
\end{figure}
\begin{figure}[h!]
\begin{center}
\includegraphics[scale=0.27]{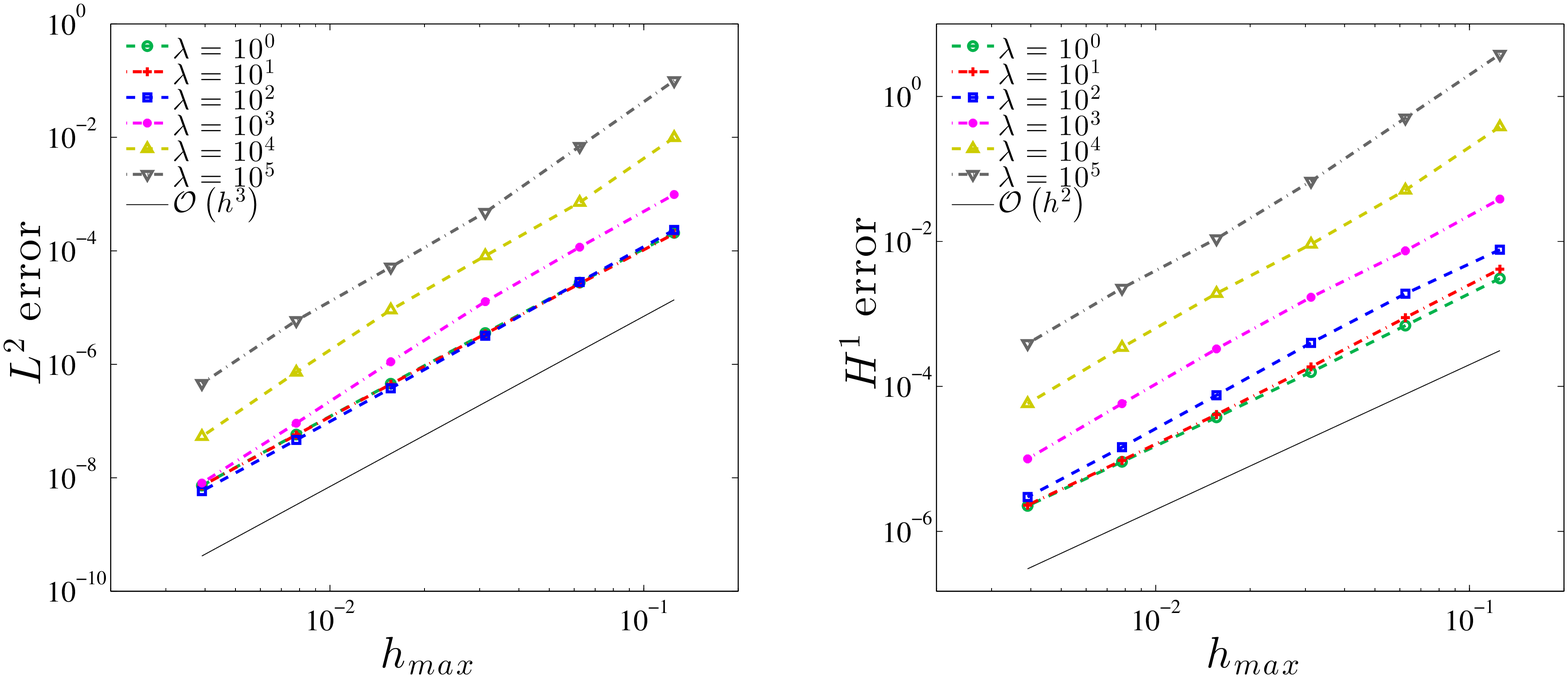}
\caption{Compressible elasticity, $V_h^2$: error versus the maximal element diameter $h_{max}$. Left: $L^2$-error, right: $H^1$-error.}
\label{linear_elasticity_graph_D2}
\end{center}
\end{figure}

The piecewise affine case (Figure \ref{linear_elasticity_graph_D1}) shows locking for $\lambda=10^5$. When $\lambda$ becomes too large, the convergence of the error does not hold if $h_{max}$ is not small enough. When the piecewise quadratic approximation is used (Figure \ref{linear_elasticity_graph_D2}), the problem with large values of $\lambda$ only changes the value of the error constant and has negligible effect on the observed rates of convergence. The numerical results show that for both cases the rate of convergence of the $H^1$-error corresponds to what has been shown theoretically. For the $L^2$-error, we observe a convergence of order ${\rm O}\left(h^{k+1}\right)$, which is a super convergence with $O(h^{1/2})$ compared to the theoretical result. In spite of numerous numerical experiments not reported here, we have not been able to find an example
exhibiting the suboptimal $L^2$-convergence of Proposition \ref{L2stabelast}.

% last sentence added correction review 1
\subsection{Incompressible elasticity}
The manufactured solution considered in this part defines the velocity and the pressure respectively such that
$$\boldsymbol{u}=\begin{pmatrix}\text{sin}(4\pi x)\text{cos}(4\pi y)\\-\text{cos}(4\pi x)\text{sin}(4\pi y)\end{pmatrix}~,~~~~~~~~~p=\pi\text{cos}(4\pi x)\text{cos}(4\pi y).$$
The nonsymmetric Nitsche's method without penalty given by equation (\ref{formulationstokes}) is used to compute approximations on a series of structured meshes. We take $\mu=1$, a range of values of $\gamma$ has been considered in the tests to study numerically the effect of the stabilization parameter on the computational error.
\begin{figure}[h!]
\begin{center}
\includegraphics[scale=0.27]{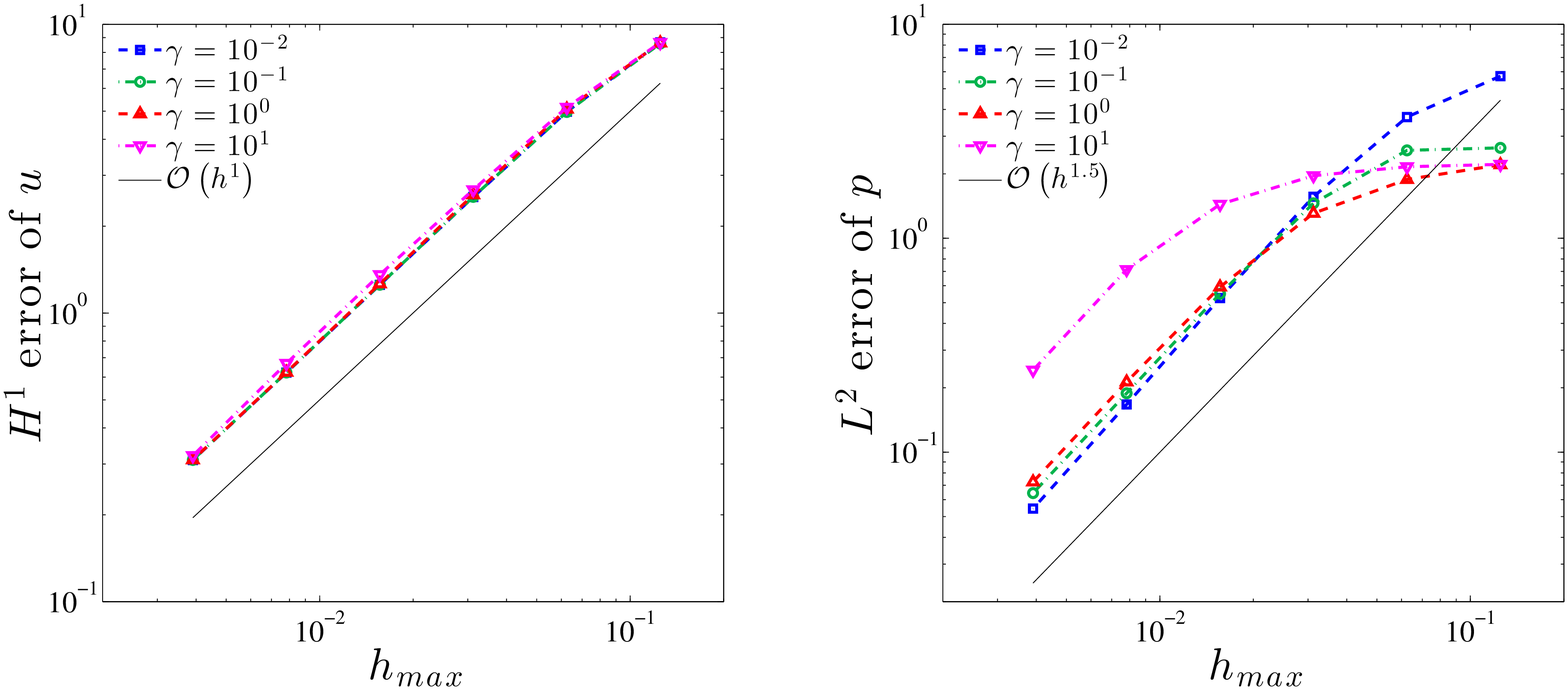}
\caption{Incompressible elasticity, $V_h^1\times Q_h^1$: errors for a range of value of $\gamma$ versus the maximal element diameter $h_{max}$. Left: $H^1$-error of the velocity, right : $L^2$-error of the pressure.}
\label{stokes_S1}
\end{center}
\end{figure}
Figure $\ref{stokes_S1}$ considers piecewise affine approximation. It shows that in this case the $H^1$-error of the velocity has an order of convergence ${\rm O}\left(h^{1}\right)$ for all the values of $\gamma$ tested. The convergence rates for the $L^2$-error of the pressure are close to ${\rm O}\left(h^{3/2}\right)$ for all the values of $\gamma$ considered and for $h_{max}$ small enough.

% added correction review 1
\subsection{Cook's membrane problem}
The Cook's membrane problem is a bending dominated test case. Figure \ref{cook_plan} represents the computational domain $\Omega$. On the face ($CD$) the Dirichlet boundary condition $\boldsymbol{u}=0$ is imposed. On the face ($AC$) the Neumann boundary condition $\boldsymbol{\sigma}(\boldsymbol{u})=(0,100)$ is imposed.

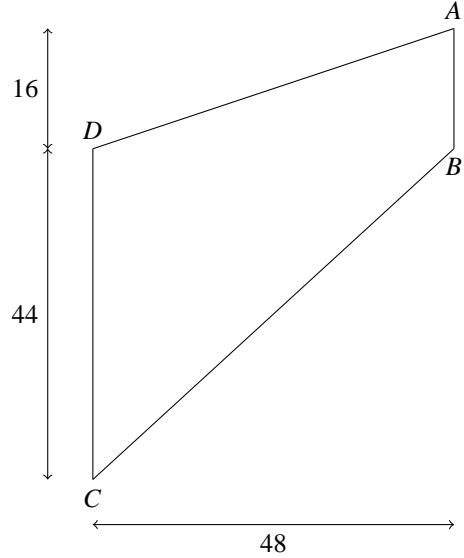
\begin{figure}[h!]
    \begin{center}
\begin{tikzpicture}[scale=0.1]

    \draw (0,0)  -- (0,44);
    \draw (0,44)  -- (48,60);
    \draw (48,60)  -- (48,44);
    \draw (48,44)  -- (0,0);
    
	\draw [<->] (0,-6)--(48,-6);
	\draw [<->] (-6,0)--(-6,44);
	\draw [<->] (-6,44)--(-6,60);

\draw (24,-6)node[below]{$48$};
\draw (-6,22)node[left]{$44$};
\draw (-6,52)node[left]{$16$};
\draw (0,0)node[below]{$C$};
\draw (0,44)node[above]{$D$};
\draw (48,60)node[above]{$A$};
\draw (48,44)node[below]{$B$};

\end{tikzpicture}
\end{center}
\caption{Cook's membrane, computational domain.}
    \label{cook_plan}
\end{figure}

In this part we compare the results given by the strong and weak imposition of the Dirichlet boundary condition. The weak imposition is implemented using the nonsymmetric Nitsche's method without penalty. We use first and second order polynomial approximation on unstructured meshes. For the first test $E=10^5$ and $\nu=0.3333$, we use compressible elasticity, note that ${\rm O}\left(\mu\right)={\rm O}\left(\lambda\right)$ ($\mu=37501$, $\lambda=74979$) . Figure \ref{deformed_mesh} shows the deformed mesh obtained.

\begin{figure}[h!]
\begin{center}
    \includegraphics[scale=0.4]{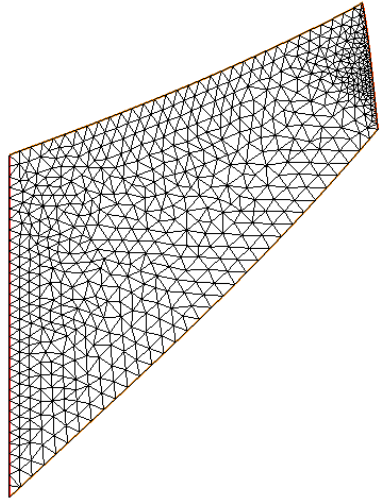}
\caption{Deformed mesh, with a magnification factor of 10.}
\label{deformed_mesh}
\end{center}
\end{figure}

We computte the vertical displacement of the point $A$ (top corner) versus the meshsize. Figure \ref{cvg_mu=lambda} shows the results for this case, by refining the mesh the approximation of the displacement of $A$ becomes more accurate. Both weak and strong imposition of the Dirichlet boundary are displayed. For first and second order approximation the weak imposition case converges faster than the strong imposition.

\begin{figure}[h!]
\begin{center}
    \includegraphics[scale=0.32]{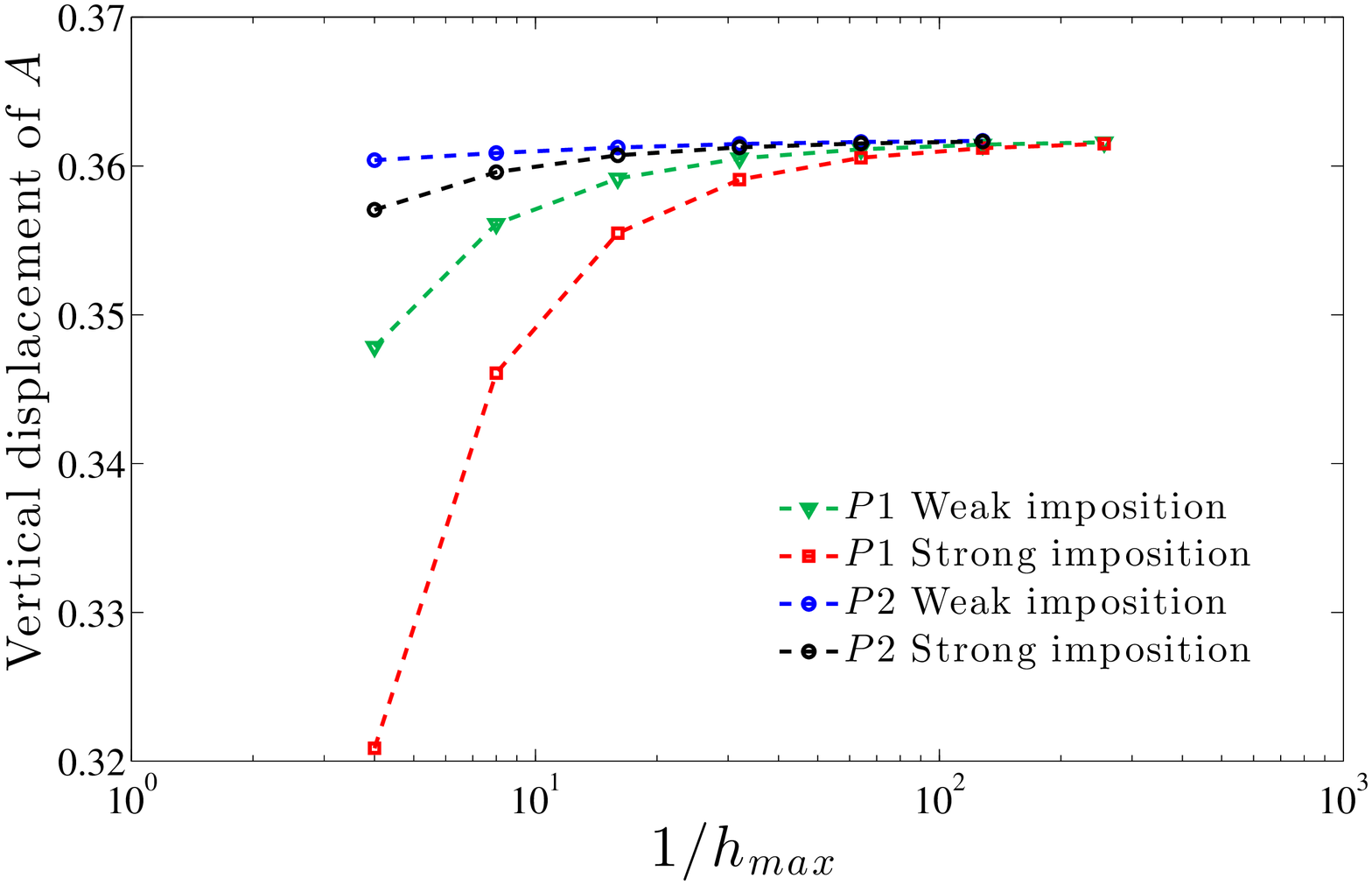}
\caption{Convergence of the vertical displacement, $E=10^5$ $\nu=0.3333$.}
\label{cvg_mu=lambda}
\end{center}
\end{figure}

For the second test we consider $E=250$ and $\nu=0.4999$, we expect to observe locking as ${\rm O}\left(\mu\right)\ll{\rm O}\left(\lambda\right)$ ($\mu=83$, $\lambda=416610$). Using compressible elasticity we perform the same tests as for the first study.

\begin{figure}[h!]
\begin{center}
    \includegraphics[scale=0.32]{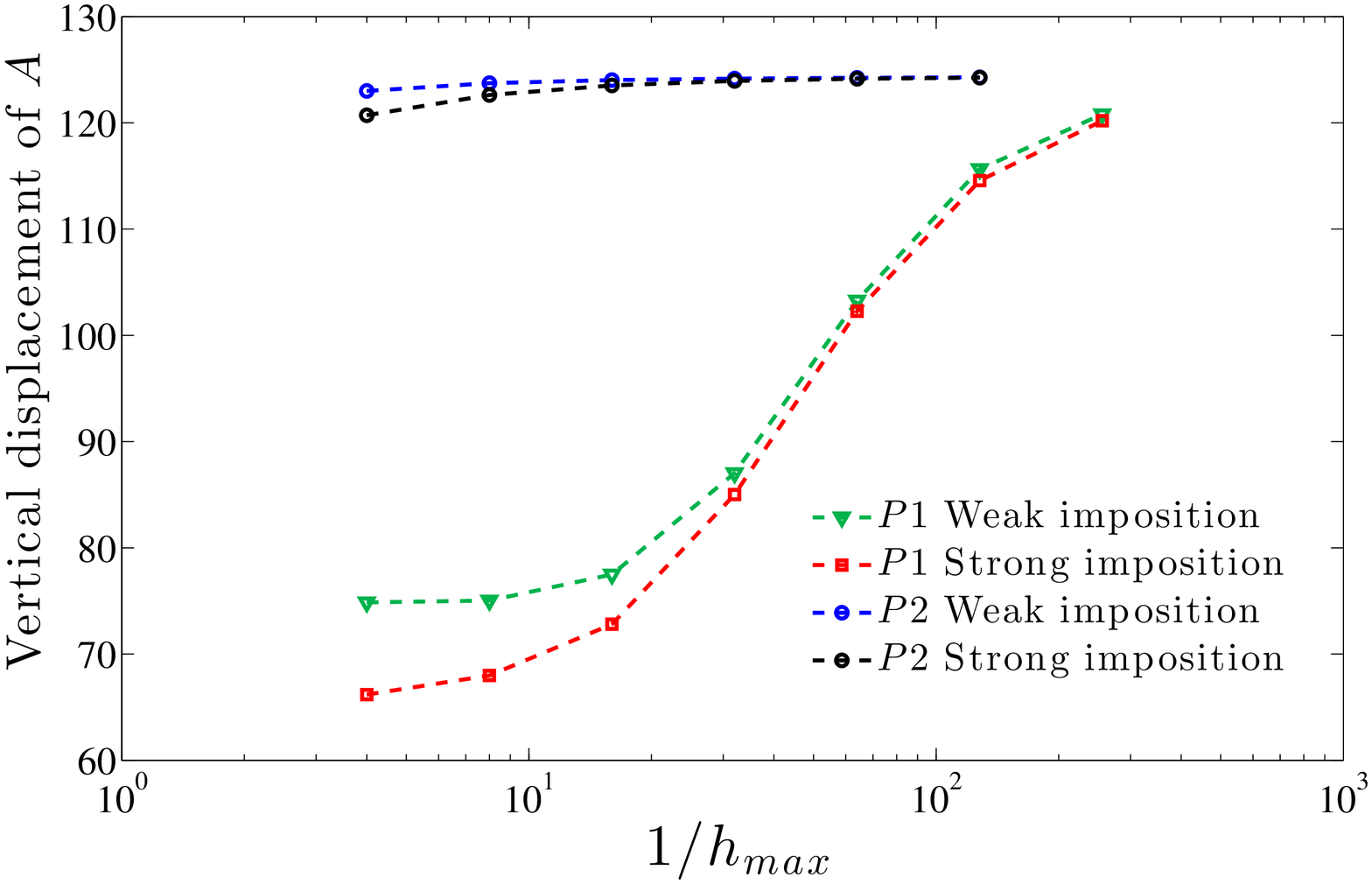}
\caption{Convergence of the vertical displacement, $E=250$ $\nu=0.4999$.}
\label{cvg_biglambda_elast}
\end{center}
\end{figure}

Figure \ref{cvg_biglambda_elast} represents the vertical displacement of the point $A$ (top corner) versus the meshsize. We observe locking for both methods for first order approximation. The second order approximation converges without locking even for the coarse meshes. Similarly as the previous case the convergence is faster for the weak imposition. In view of the observed locking, we use the nearly incompressible problem to perform the same computations. The nearly incompressible problem, is obtained considering (\ref{stokes}) and replacing $\nabla \cdot \boldsymbol{u}=0$ by $\nabla \cdot \boldsymbol{u}=p/\lambda$.

\begin{figure}[h!]
\begin{center}
    \includegraphics[scale=0.32]{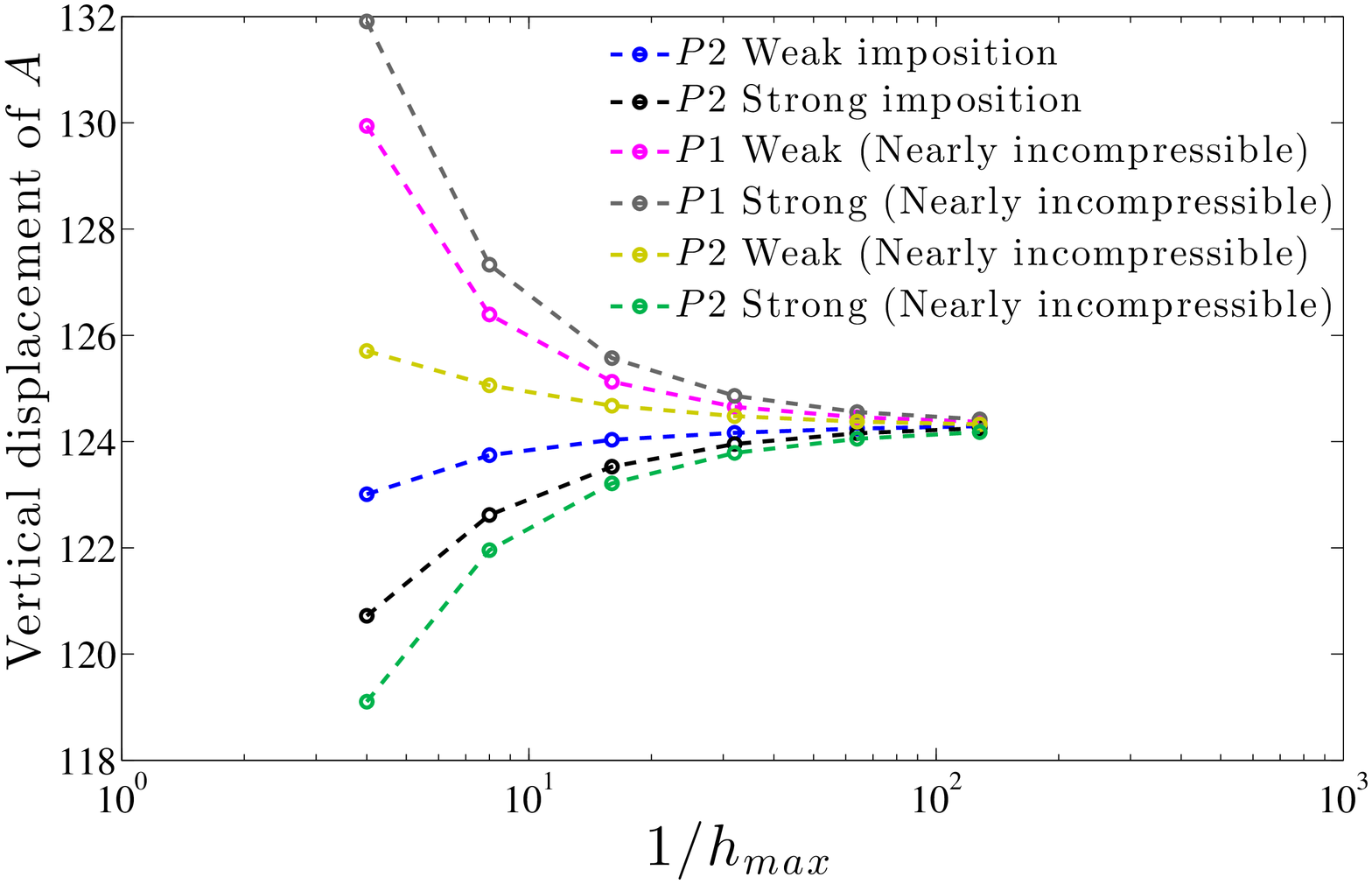}
\caption{Convergence of the vertical displacement, $E=250$ $\nu=0.4999$.}
\label{cvg_biglambda_all}
\end{center}
\end{figure}

Figure \ref{cvg_biglambda_all} displays the nearly incompressible elasticity for first and second order approximations for the weak and strong imposition but also the compressible elasticity with second order approximation. It shows that for nearly incompressible elasticity there is no locking for the method using first order polynomial approximation however for second order approximation the compressible elasticity converges faster than the nearly incompressible elasticity. Once again the weak imposition case converges faster than the strong imposition.

\section*{Appendix}
\subsection*{Proof of Lemma \ref{prop_patch}}
\begin{itemize}
\item \eqref{stdapprox}\\
There exists $x_0\in F_j$ such that $(\boldsymbol{u}_h-\overline{\boldsymbol{u}}_h^j)(x_0)=0$, then for $x\in F_j$
$$(\boldsymbol{u}_h-\overline{\boldsymbol{u}}_h^j)(x)=\int_{x_0}^x \nabla \boldsymbol{u}_h\cdot \boldsymbol{\tau}~\text{d}s,$$
using the Cauchy-Schwarz inequality it follows that
$$\left\|\boldsymbol{u}_h-\overline{\boldsymbol{u}}_h^j\right\|_{F_j}
\lesssim\left(\int_{F_j}\left(\int_{F_j} \left|\nabla \boldsymbol{u}_h\cdot \boldsymbol{\tau}\right|~\text{d}s\right)^2\text{d}s\right)^\frac12\lesssim h^\frac12\left\|\nabla \boldsymbol{u}_h\cdot \boldsymbol{\tau}\right\|_{F_j}\left(\int_{F_j}\text{d}s\right)^\frac12.$$
\item \eqref{inequality4}\\
The triangle inequality gives
$$\left\|h^{-\frac12}\boldsymbol{u}_h\right\|_{F_j}^2\leq \left\|h^{-\frac12}\left(\boldsymbol{u}_h-\bar{\boldsymbol{u}}_h\right)\right\|_{F_j}^2+\left\|h^{-\frac12}\bar{\boldsymbol{u}}_h\right\|_{F_j}^2,$$
considering the inequality (\ref{stdapprox}) and the trace inequality we can write
$$\left\|\left(\boldsymbol{u}_h-\bar{\boldsymbol{u}}_h\right)\right\|_{F_j}\lesssim h^{\frac12}\left\|\nabla \boldsymbol{u}_h\right\|_{P_j}.$$
\item \eqref{inequality7}\\
Applying the Poincar\'e inequality, on each patch $P_j$ the inequality follows.
\item \eqref{keyineq1}, \eqref{keyineq2}\\
Using the properties of $\boldsymbol{v}_j$ \eqref{defvgamma}, \eqref{prop_vj} and the Lemma 4.1 of \cite{Burman_2012_b}.
\end{itemize}

\subsection*{Proof of Lemma \ref{inequality6}}

\begin{proof}
In the rotated frame $(\xi,\eta)$, applying the definition of the $P_0$-projection, we can write the bilinear form as
\begin{eqnarray*}
\left\langle\lambda\hat{\nabla} \cdot \hat{ \boldsymbol{v}}_j, \hat{ \boldsymbol{u}}_h \cdot  \hat{\boldsymbol{n}}\right\rangle_{\hat{F}_j}
&=&\lambda\int_{\hat{F}_j}  \left(\alpha_1\frac{\partial \hat{v}_1}{\partial \xi}+\alpha_2\frac{\partial \hat{v}_2}{\partial \eta}\right)\hat{u}_2~\text{d}\hat{s}\\
&=&\lambda\int_{\hat{F}_j} \alpha_1\frac{\partial \hat{v}_1}{\partial \xi}\hat{u}_2+ \alpha_2\frac{1}{h}\left(P_0\hat{u}_2\right)^2~\text{d}\hat{s}+\lambda\int_{\hat{F}_j}\alpha_2\frac{\partial \hat{v}_2}{\partial \eta}\left(\hat{u}_2-P_0\hat{u}_2\right)~\text{d}\hat{s}.
\end{eqnarray*}
We observe that $\frac{\partial \hat{v}_1}{\partial \xi}=\hat{\nabla} \cdot(\hat{v}_1,0)^{\rm T}.$
Using the trace inequality, the inverse inequality and \eqref{keyineq1} \eqref{keyineq2}, we can show
$$\left\|\frac{\partial \hat{v}_1}{\partial \xi}\right\|_{\hat{F}_j}\lesssim h^{-1}\left\|\overline{\boldsymbol{u}}_h^j\cdot \boldsymbol{\tau}\right\|_{F_j}.$$
Note that $\int_{\hat{F}_j}\frac{\partial \hat{v}_1}{\partial \xi}~\text{d}\hat{s}=0$, using these properties and the inequality (\ref{stdapprox}), it follows that
\begin{eqnarray*}
\lambda\int_{\hat{F}_j}\alpha_1\frac{\partial \hat{v}_1}{\partial \xi}\hat{u}_2~\text{d}\hat{s}
&=&\lambda\int_{\hat{F}_j}\alpha_1\frac{\partial \hat{v}_1}{\partial \xi}\left(\hat{u}_2-P_0\hat{u}_2\right)~\text{d}\hat{s}\\
&\geq&-C\alpha_1 h^{-1}\left\|\lambda^\frac12 \overline{\boldsymbol{u}}_h^j\cdot\boldsymbol{\tau}\right\|_{F_j}\left\|\lambda^\frac12 \left(\boldsymbol{u}_h-\overline{\boldsymbol{u}}_h^j\right)\cdot \boldsymbol{n}\right\|_{F_j}\\
&\geq&-\frac{C\alpha_1^2}{4\epsilon}\left\|\frac{\lambda^\frac12}{h^\frac12}\overline{\boldsymbol{u}}_h^j\cdot\boldsymbol{\tau}\right\|_{F_j}^2-\epsilon\left\|\lambda^\frac12\nabla \boldsymbol{u}_h\right\|_{P_j}^2.
\end{eqnarray*}
Using (\ref{defvgamma}) we can obtain similarly
\begin{eqnarray*}
\lambda\int_{\hat{F}_j}\alpha_2\frac{\partial \hat{v}_2}{\partial \eta}\left(\hat{u}_2-P_0\hat{u}_2\right)~\text{d}\hat{s}&\geq&-\frac{C\alpha_2^2}{4\epsilon}\left\|\frac{\lambda^\frac12}{h^\frac12}\overline{\boldsymbol{u}}_h^j\cdot \boldsymbol{n}\right\|_{F_j}^2-\epsilon\left\|\lambda^\frac12\nabla \boldsymbol{u}_h\right\|_{P_j}^2,\\
\lambda\int_{\hat{F}_j}\alpha_2\frac{1}{h}\left(P_0\hat{u}_2\right)^2~\text{d}\hat{s}&=&\alpha_2\left\|\frac{\lambda^\frac12}{h^\frac12}\overline{\boldsymbol{u}}_h^j\cdot  \boldsymbol{n}\right\|_{F_j}^2.
\end{eqnarray*}
\end{proof}\\
\subsection*{Proof of Lemma \ref{inequality5}}
\begin{proof}
In the rotated frame $(\xi,\eta)$, applying the definition of the $P_0$-projection, we can write the bilinear form similarly as in the previous proof
\begin{eqnarray*}
\left\langle 2\mu\hat{\boldsymbol{\varepsilon}}\left(\hat{\boldsymbol{v}}_j\right) \cdot \hat{\boldsymbol{n}},\hat{\boldsymbol{u}}_h \right\rangle_{\hat{F}_j}
&=&\mu\int_{\hat{F}_j}\alpha_1\frac{\partial \hat{v}_1}{\partial\eta}\hat{u}_1+\alpha_2\frac{\partial \hat{v}_2}{\partial \xi}\hat{u}_1+2\alpha_2\frac{\partial \hat{v}_2}{\partial\eta}\hat{u}_2~\text{d}\hat{s}\\
&=&\mu\int_{\hat{F}_j}\alpha_1\frac{1}{h}\left(P_0\hat{u}_1\right)^2+\alpha_2\frac{\partial \hat{v}_2}{\partial \xi}\hat{u}_1+\alpha_2\frac{2}{h}\left(P_0\hat{u}_2\right)^2~\text{d}\hat{s}\\
&&+\mu\int_{\hat{F}_j}\alpha_1\frac{\partial \hat{v}_1}{\partial\eta}\left( \hat{u}_1-P_0\hat{u}_1\right)~\text{d}\hat{s}+2\mu\int_{\hat{F}_j}\alpha_2\frac{\partial \hat{v}_2}{\partial\eta}\left( \hat{u}_2-P_0\hat{u}_2\right)~\text{d}\hat{s}.
\end{eqnarray*}
Term by term we obtain
\begin{eqnarray*}
 \mu\int_{\hat{F}_j}\alpha_1\frac{1}{h}\left(P_0\hat{u}_1\right)^2~\text{d}\hat{s}&=&\alpha_1\left\|\frac{\mu^\frac12}{h^\frac12}\overline{\boldsymbol{u}}_h^j\cdot \boldsymbol{\tau}\right\|_{F_j}^2,\\
 \mu\int_{\hat{F}_j}\alpha_2\frac{2}{h}\left(P_0\hat{u}_2\right)^2~\text{d}\hat{s}&=&2\alpha_2\left\|\frac{\mu^\frac12}{h^\frac12}\overline{\boldsymbol{u}}_h^j\cdot \boldsymbol{n}\right\|_{F_j}^2,\\
\mu\int_{\hat{F}_j}\alpha_1\frac{\partial \hat{v}_1}{\partial \eta}\left(\hat{u}_1-P_0\hat{u}_1\right)~\text{d}\hat{s}&\geq&-\frac{C\alpha_1^2}{4\epsilon}\left\|\frac{\mu^\frac12}{h^\frac12}\overline{\boldsymbol{u}}_h^j\cdot \boldsymbol{\tau}\right\|_{F_j}^2-\epsilon\left\|\mu^\frac12\nabla \boldsymbol{u}_h\right\|_{P_j}^2,\\
2\mu\int_{\hat{F}_j}\alpha_2\frac{\partial \hat{v}_2}{\partial \eta}\left(\hat{u}_2-P_0\hat{u}_2\right)~\text{d}\hat{s}&\geq&-\frac{C\alpha_2^2}{\epsilon}\left\|\frac{\mu^\frac12}{h^\frac12}\overline{\boldsymbol{u}}_h^j\cdot \boldsymbol{n}\right\|_{F_j}^2-\epsilon\left\|\mu^\frac12\nabla \boldsymbol{u}_h\right\|_{P_j}^2.
\end{eqnarray*}
We observe that $\frac{\partial \hat{v}_2}{\partial \xi}=\hat{\nabla} (0,\hat{v}_2)^{\rm T}\cdot \boldsymbol{\tau}.$
Using the trace inequality, the inverse inequality and \eqref{keyineq1} \eqref{keyineq2}, we can show
$$\left\|\frac{\partial \hat{v}_2}{\partial \xi}\right\|_{\hat{F}_j}\lesssim h^{-1}\left\|\overline{\boldsymbol{u}}_h^j\cdot \boldsymbol{n}\right\|_{F_j}.$$
Note that since $\int_{\hat{F}_j}\frac{\partial \hat{v}_2}{\partial \xi}~\text{d}\hat{s}=0$, we obtain
\begin{equation*}
\mu\int_{\hat{F}_j}\alpha_2\frac{\partial \hat{v}_2}{\partial \xi}\hat{u}_1~\text{d}\hat{s}
=\mu\int_{\hat{F}_j}\alpha_2\frac{\partial \hat{v}_2}{\partial \xi}\left(\hat{u}_1-P_0\hat{u}_1\right)~\text{d}\hat{s}\geq-\frac{C\alpha_2^2}{4\epsilon}\left\|\frac{\mu^\frac12}{h^\frac12}\overline{\boldsymbol{u}}_h^j\cdot\boldsymbol{n}\right\|_{F_j}^2-\epsilon\left\|\mu^\frac12\nabla \boldsymbol{u}_h\right\|_{P_j}^2.
\end{equation*}
\end{proof}

\section*{Acknowledgement}

This work received funding from EPSRC (award number EP/J002313/1) which is gratefully acknowledged. We also thank an anonymous reviewer for suggesting the study of the Cook's membrane.

\bibliographystyle{IMANUM-BIB}
\bibliography{Bibliography}

\end{document}